\documentclass[10pt,reqno]{amsart}
\usepackage{amssymb,amsrefs,mathrsfs,latexsym}
\usepackage{mathtools}
\usepackage{enumerate}
\theoremstyle{plain}
\newtheorem{theorem}{Theorem}
\newtheorem{corollary}{Corollary}

\newtheorem{lemma}{Lemma}

\newtheorem*{TA}{Theorem A}
\newtheorem*{TB}{Theorem B}
\newtheorem*{TC}{Theorem C}

\theoremstyle{definition}

\DeclarePairedDelimiter{\ceil}{\lceil}{\rceil}
\DeclarePairedDelimiter{\floor}{\lfloor}{\rfloor}

\theoremstyle{remark}

\newtheorem{remark}{Remark}
\numberwithin{equation}{section}

\begin{document}

\title[Regularity and decay of fifth order KdV]{Propagation of regularity and persistence of decay for fifth order dispersive models}
\author{Jun-ichi Segata}
\address[J. Segata]{Mathematical Institute\\
Tohoku University\\
Aoba, Sendai 980-8578\\
Japan.}
\email{segata@math.tohoku.ac.jp}

\author{Derek L. Smith}
\address[D. L. Smith]{Department  of Mathematics\\
University of California\\
Santa Barbara, CA 93106\\
USA.}
\email{dls@math.ucsb.edu}
\begin{abstract} 
This paper considers the initial value problem for a class of fifth order dispersive models containing the fifth order KdV equation
\begin{equation*}
\partial_tu - \partial_x^5u
	- 30u^2\partial_xu + 20\partial_xu\partial_x^2u + 10u\partial_x^3u = 0.
\end{equation*}
The main results show that regularity or polynomial decay of the data on the positive half-line yields regularity in the solution for positive times.
\end{abstract}
\maketitle

\begin{section}{Introduction}\label{S:1}

In this work we study propagation of regularity and persistence of decay results for a class of fifth order dispersive models. For concreteness, the main theorems are stated for initial value problems of the form
\begin{equation} \label{KDV5G}
\begin{cases}
	\partial_tu - \partial_x^5u
		+ c_1u^2\partial_xu + c_2\partial_xu\partial_x^2u + c_3u\partial_x^3u = 0,
		\qquad x,t\in\mathbb{R}, \\
	u(x,0) = u_0(x),
\end{cases}
\end{equation}
where $c_{j}$ are real constants, 
$u:\mathbb{R}\times\mathbb{R}\to\mathbb{R}$ 
is an unknown function and $u_0:\mathbb{R}\to
\mathbb{R}$ is a given function. 
Eq. (\ref{KDV5G}) contains the specific equation
\begin{equation} \label{KDV5}
\partial_tu - \partial_x^5u
	- 30u^2\partial_xu + 20\partial_xu\partial_x^2u + 10u\partial_x^3u = 0
\end{equation}
which is the third equation in the sequence of nonlinear dispersive equations
\begin{equation} \label{KDVH}
\partial_tu + \partial_x^{2j+1}u + Q_j(u,\partial_xu,\dots,\partial_x^{2j-1}u) = 0,
	\quad j\in\mathbb{Z}^+,
\end{equation}
known as the KdV hierarchy. Here the polynomials $Q_j$ are chosen so that equation \eqref{KDVH} has the Lax pair formulation
\begin{equation*}
\partial_tu = [B_j;L]u
\end{equation*}
for $L=\frac{d^2}{dx^2}-u(x)$ the Schr\"odinger operator \cite{MR0235310}. The first two equations in the hierarchy are
\begin{equation}
\partial_tu - \partial_xu = 0
\end{equation}
and the KdV equation
\begin{equation} \label{KDV}
\partial_tu + \partial_x^3u + u\partial_xu = 0.
\end{equation}

With only slight modifications concerning the hypothesis on the initial data, the techniques in this paper apply to a large class of fifth order equations including the following models arising from mathematical physics:
\begin{equation} \label{E2}
\partial_tu + \partial_xu + c_1u\partial_xu + c_2\partial_x^3u
	+ c_3\partial_xu\partial_x^2u + c_4u\partial_x^3u + c_5\partial_x^5u = 0
\end{equation}
modelling the water wave problem for long, small amplitude waves over shallow bottom \cite{MR755731}, a model describing short and long wave interaction \cite{MR0463715}
\begin{equation} \label{E3}
\partial_tu - 2\partial_xu\partial_x^2u - u\partial_x^3u + \partial_x^5u = 0,
\end{equation}
and Lisher's model for motion of a lattice of anharmonic oscillators \cite{Lisher1974}
\begin{equation} \label{E4}
\partial_tu + (u+u^2)\partial_xu
	+ (1+u)(\partial_xu\partial_x^2u + u\partial_x^3u) + \partial_x^5u = 0.
\end{equation}
See also \cite{MR1216734} and references therein.

Following Kato's definition \cite{MR759907}, the 
initial value problem (IVP) 
\eqref{KDV5G} is said to be {\em locally well-posed} in the Banach space $X$ if for every $u_0 \in X$ there exists $T>0$ and a unique solution $u(t)$ satisfying
\begin{equation} \label{LWP}
u \in C([0,T] ; X) \cap Y_T,
\end{equation}
where $Y_T$ is an auxillary function space. Moreover, the solution map $u_0 \mapsto u$ is continuous 
from $X$ into the class \eqref{LWP}. If $T$ can be taken arbitrarily large, the IVP \eqref{KDV5G} is said to be {\em globally well-posed}. The persistence condition \eqref{LWP} states that the solution curve describes a dynamical system.


It is natural to study the IVP \eqref{KDV5G} in the Sobolev spaces
\begin{equation*}
H^s(\mathbb{R}) = (1-\partial_x^2)^{-s/2}L^2(\mathbb{R}),	\quad	s\in\mathbb{R},
\end{equation*}
having norm
\begin{equation*}
\|f\|_{H^s} = \|J^sf\|_2 \sim \|f\|_2 + \|D^sf\|_2.
\end{equation*}
The homogeneous derivative $D$ and its inhomogeneous counterpart $J$ are defined via the Fourier multipliers
\begin{equation*}
\widehat{D^sf}(\xi) = |\xi|^{s}\hat{f}(\xi)
	\qquad\text{and}\qquad
	\widehat{J^sf}(\xi) = \langle\xi\rangle^{s}\hat{f}(\xi),
	\quad s\in\mathbb{R},
\end{equation*}
where $\langle x \rangle = (1+x^2)^{1/2}$. The weighted spaces
\begin{equation*}
X_{s,m} = H^s(\mathbb{R}) \cap L^2(|x|^m \; dx)
	\quad s\in\mathbb{R}, m\in\mathbb{Z}^+\cup\{0\}
\end{equation*}
also appear in our analysis. Additionally, we use the notation $x_+=\max\{0,x\},x_{-}=\min\{0,x\}$ and write $A \lesssim B$ to denote $A \leq cB$ when the value of the fixed constant $c$ is immaterial. The floor and ceiling functions are denoted by $\floor{x}$ and $\ceil{x}$, respectively.


The persistence property \eqref{LWP} doesn't preclude all smoothing effects. For step-data, Murray 
\cite{MR0470533} proved the existence of solutions to the initial value problem for the KdV equation \eqref{KDV} in the class $C^\infty(\{x,t : x\in\mathbb{R}, \; t>0)\})$ which weakly recover the initial data. T. Kato \cite{MR759907} described this {\em quasiparabolic smoothing} effect as stemming from the unidirectional dispersion inherent in the equation. He obtained a similar result for data having exponential decay on the positive half-line. 
The Kato estimates occur in the asymmetric spaces
\begin{equation*}
H^s(\mathbb{R}) \cap L_\beta^2(\mathbb{R}),
	\quad s\geq0,\ \beta>0,
\end{equation*}
where
\begin{equation*}
L_\beta^2(\mathbb{R}) = L^2(e^{\beta x} \; dx),
\end{equation*}
in which the operator $\partial_t+\partial_x^3$ is formally equivalent to $\partial_t+(\partial_x-\beta)^3$. The use of asymmetric spaces leads to a result which is irreversible in time. Isaza, Linares and Ponce \cite{MR3279353} extended the quasiparabolic smoothing effect to a large class of fifth order equations.
\begin{TA}$($Isaza, Linares and Ponce \cite{MR3279353}$)$
Let $u \in C([0,T];H^6(\mathbb{R}))$ be a solution of the IVP associated to the equation
\begin{equation}
\partial_tu - \partial_x^5u + Q_0(u,\partial_xu,\partial_x^2u)\partial_x^3u + Q_1(u,\partial_xu,\partial_x^2u)= 0
\end{equation}
corresponding to initial data $u_0 \in H^6(\mathbb{R}) \cap L^2(e^{\beta x} \; dx), \beta>0$, with
\begin{equation}
Q_0 = \sum_{1 \leq i+j+k \leq N} a_{i,j,k} u^i(\partial_xu)^j(\partial_x^2u)^k,
	\quad N\in\mathbb{Z}^+, N \geq 1,\  a_{i,j,k}\in\mathbb{R}
\end{equation}
and
\begin{equation}
Q_1 = \sum_{2 \leq i+j+k \leq M} b_{i,j,k} u^i(\partial_xu)^j(\partial_x^2u)^k
	\quad M\in\mathbb{Z}^+, M \geq 2,\ 
	 b_{i,j,k}\in\mathbb{R}.
\end{equation}
Then
\begin{equation*}
e^{\beta x}u \in C([0,T];L^2(\mathbb{R})) \cap C((0,T);H^\infty(\mathbb{R})),
\end{equation*}
and
\begin{equation*}
\|e^{\beta x}u(t)\|_2 \leq c\|e^{\beta x}u_0\|_2,
	\quad t\in[0,T].
\end{equation*}
\end{TA}

T. Kato \cite{MR759907} 
demonstrated the existence of weak global 
solutions $u$ to the KdV equation \eqref{KDV} 
corresponding to initial data in $L^2(\mathbb{R})$. 
A key step in his proof is the a priori estimate of $\|u\|_{H^1(-R,R)}$ in terms of $\|u_0\|_2$. In addition, his approach shows the following {\em local smoothing} effect.
\begin{TB}$($T. Kato \cite{MR759907}$)$
Let $s>3/2$ and $0<T<\infty$. If $u \in C([0,T];
H^s(\mathbb{R}))$ is the solution to \eqref{KDV}, then
\begin{equation*}
u \in L^2([0,T];H^{s+1}(-R,R))
	\quad \text{for any $0<R<\infty$,}
\end{equation*}
with the associated norm depending only on $\|u_0\|_{H^s}$, $R$ and $T$.
\end{TB}
Roughly, the proof follows by observing that a smooth solution $u$ to the IVP associated to the KdV equation \eqref{KDV} satisfies the identity
\begin{align} \label{KATO}
&\frac{d}{dt} \int (\partial_x^ku)^2\psi \; dx + 3 \int (\partial_x^{k+1}u)^2\psi' \; dx \notag \\
	&\qquad\qquad
	= \int (\partial_x^ku)^2\psi''' \; dx
		+ \int \partial_x(\psi u)(\partial_x^ku)^2 \; dx
		+ \int \partial_x^ku [\partial_x^k;u]\partial_xu\psi \; dx.
\end{align}
for $k\in\mathbb{Z}^+$. Selecting $\psi=\psi(x)$ to be a sufficiently smooth, nonnegative, nondecreasing cutoff function, integration of the above identity in time yields local estimates of $\partial_x^{k+1}u$ as each term on right-hand side can be controlled by $\|u\|_{L_T^\infty H^k}$.

Isaza, Linares and Ponce applied Kato's argument to study the propagation of regularity and persistence of decay of solutions to the  $k$-generalized KdV and Benjamin-Ono equations in \cite{IsazaLinaresPonce2014a} and \cite{IsazaLinaresPonce2014b}, respectively. Also working in asymmetric spaces, they observed that for a solution $u$ to the KdV equation corresponding to data $u_0 \in H^{s}(\mathbb{R})$ with $s>3/4$, if $\|x^{n/2}u_0\|_{L^2(0,\infty)}$ for some $n\in\mathbb{Z}^+$, then for every $x_0\in\mathbb{R}$, $u(\cdot,t) \in H^n(x_0,\infty)$ for positive times. More succinctly, one-sided decay on the initial data yields regularity in the solution. In this paper we extend their work to fifth order dispersive models. Before stating our results we review the local well-posedness theory for \eqref{KDV5G} and related models.


Utilizing the Lax pair formulation, initial value problems associated to equations in the KdV hierarchy \eqref{KDVH} can be solved in a space of rapidly decaying functions using the inverse scattering method \cite{MR0336122}. This method does not apply to dispersive equations of a more general form.

While studying the models \eqref{KDV5G}, \eqref{E2}, \eqref{E3} and \eqref{E4}, Ponce 
\cite{MR1216734} remarked that the use of dispersive estimates appears essential to attain local well-posedness in Sobolev spaces. Using the energy method, sharp linear estimates and parabolic regularization, in \cite{MR1216734} Ponce 
proved local well-posedness for the initial value problems associated to these equations in $H^s(\mathbb{R})$, $s\geq4$ .

Kenig, Ponce and Vega investigated the class
\begin{equation} \label{KDVGH}
\begin{cases}
	\partial_tu + \partial_x^{2j+1}u + P(u,\partial_xu,\dots,\partial_x^{2j}u) = 0,
		\qquad x,t\in\mathbb{R},\\
	u(x,0) = u_0(x),
\end{cases}
\end{equation}
with $j\in\mathbb{Z}^+$ and $P : \mathbb{R}^{2j+1}\rightarrow\mathbb{R}$ (or $\mathbb{C}^{2j+1}\rightarrow\mathbb{C}$) a polynomial having no constant or linear terms. Using the contraction principle, they established in \cite{MR1321214} and \cite{MR1195480} that for a given equation in the class \eqref{KDVGH} there exists a positive real number $s_0$ and nonnegative integer $m_0$ depending only on the form of the polynomial $P$ such that the corresponding IVP is locally well-posed in the weighted space $X_{s,m}$ for all $m\in\mathbb{Z}^+$, $m \geq m_0$ and $s \geq \max\{s_0,jm\}$. Thus equations of the form \eqref{KDVGH} preserve the Schwarz class. The use of weighted spaces stems from the observation that $[L;\Gamma]=0$ for the vector fields
\begin{equation*}
L=\partial_t + \partial_x^{2j+1}
	\quad\text{and}\quad
	\Gamma = x - (2j+1)t \partial_x^{2j}.
\end{equation*}
Given that each term of $P$ has ``enough" factors, it may be that the corresponding IVP is globally well-posed, that no weight is necessary or both. For further comments, see \cite{LiPo2015}.

Following \cite{MR1885293} and \cite{MR1944575}, Pilod 
\cite{MR2446185} showed that certain initial value problems in the class \eqref{KDVGH} are in some sense ill-posed. In particular, if $P$ contains the term $u\partial_x^ku$ for $k>j$, then the solution map $H^s(\mathbb{R}) \ni u_0 \mapsto u \in C([0,T];H^s(\mathbb{R}))$ is not $C^2$ at the origin for any $s\in\mathbb{R}$. For equations of the form \eqref{KDV5G}, Kwon demonstrated that the solution map is not even uniformly continuous by using the arguments of \cite{MR2172940} and \cite{MR2350033}. All of these facts result from uncontrollable interactions when both high and low frequencies are present in the initial data. Thus, in contrast to the KdV \eqref{KDV}, equations of the form \eqref{KDV5G} cannot be solved using the contraction principle in $H^s(\mathbb{R})$.

Differences between \eqref{KDV5G} and \eqref{KDV} also arise when applying the energy estimate method. Note that after integrating by parts, smooth solutions $u$ to \eqref{KDV5G} satisfy
\begin{align} \label{LOSS}
&\frac{d}{dt} \int(\partial_x^ku)^2\psi(x) \; dx + 2 \int(\partial_x^{k+2}u)^2\psi' \; dx \notag \\
	&\qquad\qquad
	\lesssim \|\partial_x^3u\|_\infty \int(\partial_x^ku)^2\psi(x) \; dx
		+ \left|\int\partial_xu(\partial_x^{k+1}u)^2\psi \; dx\right|
		+ \cdots
\end{align}
for $k\in\mathbb{Z}^+$. After integrating in time, the right-hand side cannot be estimated in terms of $\|u\|_{L_T^\infty H^k}$. Kwon \cite{MR2455780} 
introduced a corrected energy and refined Strichartz estimate to overcome this {\em loss of derivatives} and obtained the following result.
\begin{TC} $($Kwon \cite{MR2455780}$)$ \label{THM_D}
Let $s>5/2$. For any $u_0 \in H^s(\mathbb{R})$ there exists a time $T \gtrsim \|u_0\|_{H^s}^{-10/3}$ and a unique real-valued solution $u$ for the IVP \eqref{KDV5G} satisfying
\begin{equation} \label{LWP_K5}
u \in C([0,T];H^s(\mathbb{R}))
	\qquad\text{and}\qquad
	\partial_x^3u \in L^1([0,T];L^\infty(\mathbb{R})).
\end{equation}
\end{TC}

\begin{remark}
A loss of derivatives can occur for equations for which LWP can be obtained in $H^s(\mathbb{R})$ using the contraction principle (see Section \ref{S:7}).
\end{remark}

Using an auxillary Bourgain space introduced in \cite{MR1209299} \cite{MR1215780}, 
the local well-posedness of the IVP \eqref{KDV5G} in the energy space $H^2(\mathbb{R})$ was established simultaneously by Kenig and Pilod \cite{KenigPilod2012} and Guo, Kwak and Kwon \cite{MR3096990}. Thus global well-posedness follows in the Hamiltonian case, i.e., when $c_{2}=2c_{3}$.


Our main contribution is the incorporation of Kwon's corrected energy and refined Strichartz estimate into the iterative argument used in \cite{IsazaLinaresPonce2014a} and \cite{IsazaLinaresPonce2014b}. We first describe the propagation of one-sided regularity exhibited by solutions to the IVP \eqref{KDV5G} provided by Theorem C.

\begin{theorem} \label{PROP} 
Let $s>5/2$. 
Suppose $u_0 \in H^{s}(\mathbb{R})$ and for some $l \in \mathbb{Z}^+, x_0 \in \mathbb{R}$
\begin{equation}
\|\partial_x^lu_0\|_{L^2(x_0,\infty)}^2
	= \int_{x_0}^\infty (\partial_x^lu_0)^2(x) \; dx
	< \infty.
\end{equation}
Then the solution $u$ of IVP \eqref{KDV5G} provided by Theorem C satisfies
\begin{equation}
\sup_{0 \leq t \leq T} \int_{x_0 + \epsilon - \nu t}^\infty (\partial_x^m u)^2(x,t) \; dx \leq c
\end{equation}
for any $\nu\geq0, \epsilon>0$ and each $m=0,1,\dots,l$ with
\begin{equation} \label{PCONST1}
c = c(l; \nu; \epsilon; T; \|u_0\|_{H^s}; \|\partial_x^lu_0\|_{L^2(x_0,\infty)}),
\end{equation}
where $T$ is given in Theorem C. 
In particular, for all $t\in(0,T]$, the restriction of $u(\cdot,t)$ to any interval $(x_1,\infty)$ belongs to $H^l(x_1,\infty)$.

Moreover, for any $\nu\geq0, \epsilon>0$ and $R>\epsilon$
\begin{equation} \label{PROP2}
\int_0^T \int_{x_0 + \epsilon - \nu t}^{x_0 + R - \nu t} (\partial_x^{l+2}u)^2(x,t) \; dxdt \leq \tilde{c}
\end{equation}
with
\begin{equation} \label{PCONST2}
\tilde{c} = \tilde{c}(l; \nu; \epsilon; R; T; \|u_0\|_{H^s}; \|\partial_x^lu_0\|_{L^2(x_0,\infty)}).
\end{equation}
\end{theorem}
\begin{remark}
Observe that \eqref{PROP2} is a generalization of Kato's local smoothing effect since we do not require $u_0 \in H^l(\mathbb{R})$.
\end{remark}
\begin{remark} \label{NUD}
The constants appearing in Theorem \ref{PROP} have the form of a polynomial in $\nu$. For $l\geq6$, the degree of this dependence is $d=8(l-5)$.
\end{remark}

For fixed $l\in\mathbb{Z}^+$, Theorem \ref{PROP} is the base case for the situation where the derivatives of the initial data possess polynomial decay when restricted to the positive half-line. Our second result states that this decay persists.

\begin{theorem} \label{DECAY}
Let $s>5/2$ and let $n,l\in\mathbb{Z}^+$. 
Suppose $u_0 \in H^{s}(\mathbb{R})$ and for each $m=0,1,\dots,l$
\begin{equation} \label{DDATA}
\|x^{n/2}\partial_x^mu_0\|_{L^2(0,\infty)}^2
	= \int_0^\infty x^n(\partial_x^mu_0)^2(x) \; dx
	< \infty.
\end{equation}
Then the solution $u$ of IVP \eqref{KDV5G} provided by Theorem C satisfies
\begin{equation} \label{DECAY1}
\sup_{0 \leq t \leq T} \int_\epsilon^\infty x^n(\partial_x^mu)^2(x,t) \; dx
	\leq c
\end{equation}
for any $\epsilon>0$ and each $m=0,1,\dots,l$ with
\begin{equation} \label{DCONST1}
c = c(n; l; \epsilon; T; \|u_0\|_{H^s}; \|x^{n/2}\partial_x^ku_0\|_{L^2(0,\infty)})
\end{equation}
for $k=0,1,\dots,m$, 
where $T$ is given in Theorem C. 
By local well-posedness, we may take $\epsilon=0$ for $m \leq s$.

Moreover, for any $\epsilon>0$
\begin{equation} \label{DECAY2}
\int_0^T \int_0^\infty x^{n-1}(\partial_x^{l+2}u)^2(x,t) \; dxdt \leq \tilde{c}
\end{equation}
with $\tilde{c}$ as in \eqref{DCONST1}.
\end{theorem}

The hypothesis of Theorem \ref{DECAY} may seem unneccessarily strong, but a bootstrapping argument yields regularity of the solution for positive times by imposing decay on only the initial data and not its derivatives. Thus the next theorem can be seen as a weakening of the hypothesis of Theorem A inasmuch as exponential decay implies polynomial decay.

\begin{theorem} \label{DIR}
Let $s>5/2$. Suppose $u_0 \in H^s(\mathbb{R})$ and for some $n \in \mathbb{Z}^+$
\begin{equation}
\|x^{n/2}u_0\|_{L^2(0,\infty)}^2
	= \int_0^\infty x^nu_0^2(x) \; dx
	< \infty.
\end{equation}
Then for every $\delta>0$ and any pair $m,k \in \mathbb{Z}^+\cup\{0\}$ satisfying
\begin{equation}
n=k+\lfloor m/2 \rfloor
\end{equation}
the solution $u$ of IVP \eqref{KDV5G} provided by Theorem C satisfies, for $k>0$
\begin{equation}
\sup_{\delta \leq t \leq T}
	\int_{\epsilon-\nu t}^\infty (\partial_x^mu)^2(x,t) \langle x_+\rangle^k \; dx
	+ \int_\delta^T \int_{\epsilon-\nu t}^\infty (\partial_x^{m+2}u)^2(x,t)\langle x_+\rangle^{k-1} \; dxdt
		\leq c
\end{equation}
for every $\nu\geq0, \epsilon>0$, with
\begin{equation}
c = c(n; \delta; \nu; \epsilon; T; \|u_0\|_{H^s}; \|x^{n/2}u_0\|_{L^2(0,\infty)}),
\end{equation}
where $T$ is given in Theorem C. 
For $k=0$ and any $R>\epsilon$,
\begin{equation}
\sup_{\delta \leq t \leq T}
	\int_{\epsilon-\nu t}^\infty (\partial_x^{2n}u)^2(x,t) \; dx
	+ \int_\delta^T \int_{\epsilon-\nu t}^{R-\nu t} (\partial_x^{2n+2}u)^2(x,t) \; dxdt
		\leq \tilde{c}
\end{equation}
with $\tilde{c}$ additionally depending on $R$.
\end{theorem}

The time reversible nature of equation \eqref{KDV5G} yields a number of consequences. Combining with the contrapositive of Theorems \ref{PROP} and \ref{DIR}, we have 
the following.

\begin{corollary} 
Assume that $s>5/2$. 
Let $u \in C([-T,T];H^{s}(\mathbb{R}))$ be a solution of \eqref{KDV5G} provided by Theorem C such that
\begin{equation*}
\partial_x^mu(\cdot,\hat{t}) \notin L^2(a,\infty)
	\quad\text{for some $\hat{t}\in[-T,T]$ and 
	$a\in\mathbb{R}$}.
\end{equation*}
Then for any $t\in[-T,\hat{t})$ and any $\beta\in\mathbb{R}$
\begin{equation*}
\partial_x^mu(\cdot,t) \notin L^2(\beta,\infty)
	\quad\text{and}\quad
	x^{\ceil{m/2}/2}u(\cdot,t) \notin L^2(0,\infty).
\end{equation*}
\end{corollary}

Suppose now that the initial data has regularity to the right but also contains a singularity, for instance $u_0 \in H^{s}(\mathbb{R})$, $u_0 \notin H^l(\mathbb{R})$ and
\begin{equation*}
\partial_x^lu_0 \in L^2(b,\infty)
	\quad\text{for some $l\in\mathbb{Z}^+,l>2$.}
\end{equation*}
The persistence property \eqref{LWP} prohibits the solution from lying in $H^l(\mathbb{R})$. However, as a consequence of Remark \ref{NUD}, we deduce that for positive times $\partial_x^lu(\cdot,t)$ has only polynomial growth to the left and thus lies in $L^2_{\text{loc}}(\mathbb{R})$. That is, any singularities in $\partial_x^lu(\cdot,t)$ vanish for positive times. This is made precise by the next corollary to Theorem \ref{PROP}.

\begin{corollary} \label{NU} 
Assume that $s>5/2$. 
Let $u \in C([-T,T];H^{s}(\mathbb{R}))$ be a solution of \eqref{KDV5G} provided by Theorem C. Suppose there exists $l,m\in\mathbb{Z}^+$ with $m \leq l$ such that for some $a,b\in\mathbb{R}$ with $a<b$
\begin{equation} \label{NUH}
\int_b^\infty (\partial_x^lu_0)^2(x) \; dx < \infty
	\quad\text{but}\quad
	\partial_x^mu_0 \notin L^2(a,\infty).
\end{equation}
\begin{enumerate}[(i)]
\item For any $t\in(0,T]$ and any $\varepsilon>0$
\begin{equation} \label{NUW}
\int_{-\infty}^\infty \frac{1}{\langle x_{-} \rangle^{8(l-5)+\varepsilon}}
	(\partial_x^lu)^2(x,t) \; dx \leq c,
	\quad l\geq6
\end{equation}
with $c$ depending on $t$ and $\varepsilon$.

\item For any $t\in[-T,0)$ and any $\alpha\in\mathbb{R}$
\begin{equation*}
\int_\alpha^\infty (\partial_x^mu)^2(x,t) \; dx = \infty.
\end{equation*}
\end{enumerate}
\end{corollary}

\begin{remark}
The conclusion \eqref{NUW} holds for $l=3,4,5$ with the appropriate modification to the weight.
\end{remark}

As a consequence of Corollary \ref{NU} we see that, in general, regularity to the left does not propagate forward in time. Suppose in addition to \eqref{NUH} that
\begin{equation*}
\int_{-\infty}^a (\partial_x^lu_0)^2(x) \; dx < \infty.
\end{equation*}
If this regularity persisted we could conclude from \eqref{NUW} that $u(\cdot,t) \in H^l(\mathbb{R})$ for positive times, contradicting the persistence property \eqref{LWP}.

Beginning with Theorem \ref{DIR} yields a similar corollary.

\begin{corollary} 
Assume that $s>5/2$.
Let $u \in C([-T,T];H^{s}(\mathbb{R}))$ be a solution of \eqref{KDV5G} provided by Theorem C. If for $m,n\in\mathbb{Z}^+$, $m<n$,
\begin{equation*}
x_+^{\ceil{n/2}/2}u_0 \in L^2(0,\infty)
	\quad\text{and}\quad
	\partial_x^mu_0 \notin L^2(\beta,\infty)
	\quad\text{for some $\beta\in\mathbb{R}$},
\end{equation*}
then for any $t\in(0,T]$
\begin{equation*}
x_+^{\ceil{n/2}/2}u(\cdot,t) \in L^2(0,\infty)
	\quad\text{and}\quad
	\partial_x^nu(\cdot,t) \in L^2(\alpha,\infty)
	\quad\text{for\ any}\ \alpha\in\mathbb{R},
\end{equation*}
and for any $t\in[-T,0)$
\begin{equation*}
x_+^{\ceil{m/2}/2}u(\cdot,t) \notin L^2(0,\infty)
	\quad\text{and}\quad
	\partial_x^mu(\cdot,t) \notin L^2(\alpha,\infty)
	\quad\text{for\ any}\ \alpha\in\mathbb{R}.
\end{equation*}
\end{corollary}

Our proof technique does not rely on the particular values of the coefficients in \eqref{KDV5G}, hence Theorems \ref{PROP}, \ref{DECAY} and \ref{DIR} can be applied backwards in time. For instance, if $u(x,t)$ is a solution of \eqref{KDV5G} with regularity to the right which propagates leftward, then $u(-x,-t)$ has regularity to the left which propagates rightward. Therefore we can consider the situation when $u(\cdot,t_0)$ has decay or regularity to the right and $u(\cdot,t_1)$ has decay or regularity to the left, where $t_0<t_1$.
\begin{corollary} 
Assume that $s>5/2$. 
Let $u \in C([-T,T];H^{s}(\mathbb{R}))$ be a solution of \eqref{KDV5G} provided by Theorem C. If there exist $n_j\in\mathbb{Z}^+\cup\{0\}$, $j=1,2,3,4$, $t_0,t_1\in[-T,T]$ with $t_0<t_1$ and $a,b\in\mathbb{R}$ such that
\begin{equation*}
\int_0^\infty |x|^{n_1}|u(x,t_0)|^2 \; dx < \infty
	\quad\text{and}\quad
	\int_a^\infty |\partial_x^{n_2}u(x,t_0)|^2 \; dx < \infty
\end{equation*}
and
\begin{equation*}
\int_{-\infty}^0 |x|^{n_3}|u(x,t_1)|^2 \; dx < \infty
	\quad\text{and}\quad
	\int_{-\infty}^b |\partial_x^{n_4}u(x,t_1)|^2 \; dx < \infty
\end{equation*}
then
\begin{equation*}
u \in C([-T,T];H^s(\mathbb{R}) \cap L^2(|x|^r \; dx))
\end{equation*}
where
\begin{equation*}
s = \min\left\{\max\{2n_1,n_2\},\max\{2n_3,n_4\}\right\}
	\quad\text{and}\quad
	r=\min\{n_1,n_3\}.
\end{equation*}
\end{corollary}

In Section \ref{S:2} we construct cutoff functions which 
are needed to prove Theorems \ref{PROP}, \ref{DECAY} 
and \ref{DIR}. 
Theorems \ref{PROP} and \ref{DECAY} are proved in Sections \ref{S:3} and \ref{S:4}, respectively. In Section \ref{S:5} we prove Theorem \ref{DIR}. The proof of Corollary \ref{NU} is found in Section \ref{S:6}. We conclude in Section \ref{S:7} with an extension to a more general class of fifth order models.
\end{section}

\begin{section}{Construction of Cutoff Function}\label{S:2}

In this section we construct cutoff functions which 
are needed to prove Theorems \ref{PROP}, \ref{DECAY} 
and \ref{DIR}.  
Define the polynomial
\begin{equation*}
\rho(x) = 2772 \int_0^x y^5(1-y)^5 \; dy
\end{equation*}
which satisfies
\begin{align*}
\rho(0) &= 0,		\qquad	\rho(1)=1, \\
\rho'(0) &= \rho''(0) = \cdots = \rho^{(5)}(0) = 0, \\
\rho'(1) &= \rho''(1) = \cdots = \rho^{(5)}(1) = 0
\end{align*}
with $0<\rho,\rho'$ for $0<x<1$.
Much of the complexity of our construction airses when handling the ratio which appears in \eqref{ENERGY}, 
see Section \ref{S:3} below. Thus we note that the expression
\begin{equation} \label{RhoRatio}
\frac{(\rho'''(x))^2}{\rho'(x)}
	= -277200 x (x-1) \left(2-9 x+9 x^2\right)^2
\end{equation}
is continuous for $x\in[0,1]$ and vanishes at the endpoints.
For $\epsilon,b>0$, 
define $\chi \in C^5(\mathbb{R})$ by
\begin{equation*}
\chi(x;\epsilon,b) =
\begin{cases}
	0				& x \leq \epsilon, \\
	\rho((x-\epsilon)/b)	& \epsilon < x < b+\epsilon, \\
	1				& b + \epsilon \leq x.
\end{cases}
\end{equation*}
By construction $\chi$ is positive for $x\in(\epsilon,\infty)$ and all derivatives are supported in $[\epsilon,b+\epsilon]$. A scaling argument and \eqref{RhoRatio} provides
\begin{equation} \label{CutoffRatio}
\sup_{x\in[\epsilon,b+\epsilon]} \left|\frac{(\chi'''(x;\epsilon,b))^2}{\chi'(x;\epsilon,b)} \right| \leq c(b)
\end{equation}
and for $j=1,2,3,4,5$
\begin{equation} \label{CutoffBounded}
|\chi^{(j)}(x;\epsilon,b)| \leq c(j;b).
\end{equation}
A computation produces
\begin{equation*}
\frac{(\chi'''(x;\epsilon,b))^2}{\chi'(x;\epsilon,b)}\cdot\frac{1}{\chi'(x;\epsilon/3,b+\epsilon)}
	= q_0(x)\frac{(x-\epsilon)(b+\epsilon-x)}{(3 x-\epsilon )^5 (3 b-3 x+4 \epsilon )^5}
\end{equation*}
and for $j=1,2,3,4,5$
\begin{equation*}
\frac{\chi^{(j)}(x;\epsilon,b)}{\chi'(x;\epsilon/3,b+\epsilon)}
	= q_j(x)\frac{(x-\epsilon)(b+\epsilon-x)}{(3 x-\epsilon )^5 (3 b-3 x+4 \epsilon )^5}
\end{equation*}
where $q_0,\dots,q_5$ are polynomials. In each of the previous two cases, the right-hand side is continuous on the interval $x\in[\epsilon,b+\epsilon]$, hence bounded. These computations lead to the following estimates, which will be used in a later inductive argument:
\begin{equation} \label{CutoffRatioExpanded}
\sup_{x\in[\epsilon,b+\epsilon]} \left|\frac{(\chi'''(x;\epsilon,b))^2}{\chi'(x;\epsilon,b)}\right|
	\leq c(\epsilon;b)\chi'(x;\epsilon/3,b+\epsilon)
\end{equation}
and for $j=1,2,3,4,5$
\begin{equation} \label{CutoffExpanded}
\sup_{x\in[\epsilon,b+\epsilon]} \left|\chi^{(j)}(x;\epsilon,b)\right|
	\leq c(j;\epsilon;b)\chi'(x;\epsilon/3,b+\epsilon).
\end{equation}

Additionally, we define $\chi_n \in C^5(\mathbb{R})$ via the formula
\begin{equation*}
\chi_n(x;\epsilon,b) = x^n\chi(x;\epsilon,b).
\end{equation*}
It is helpful to make the auxillary definition
\begin{equation*}
p(y) = 462 - 1980y + 3465y^2 - 3080y^3 + 1386y^4 - 252y^5,
\end{equation*}
whose only real root occurs at $y\approx1.29727$. Note that for $n\in\mathbb{Z}^+$
\begin{equation} \label{CutoffNDerivative}
\chi_n'(x;\epsilon,b) = nx^{n-1}\chi(x;\epsilon,b) + x^n\chi'(x;\epsilon,b)
\end{equation}
which is positive for $\epsilon < x \leq b+\epsilon$. Hence the expression
\begin{equation*}
\frac{(\chi_n'''(x;\epsilon,b))^2}{\chi_n'(x;\epsilon,b)}
\end{equation*}
is continuous in this interval. To prove that it is bounded in $[\epsilon,b+\epsilon]$, we must only analyze the limit $x \rightarrow \epsilon^+$. First observe
\begin{equation*}
\chi_n'(x;\epsilon,b) = \left(\frac{x-\epsilon}{b}\right)^5
	\left( \frac{n}{b}x^{n-1}(x-\epsilon)p\left(\frac{x-\epsilon}{b}\right)
		+ \frac{2772}{b}x^n \left(1-\frac{x-\epsilon}{b}\right)^5 \right)
\end{equation*}
so that
\begin{equation*}
\lim_{x\rightarrow\epsilon^+} \frac{(\chi_n'''(x;\epsilon,b))^2}{\chi_n'(x;\epsilon,b)}
	= \left(\frac{b^6}{2772\epsilon^n}\right) \lim_{x\rightarrow\epsilon^+} \frac{(\chi_n'''(x;\epsilon,b))^2}{(x-\epsilon)^5}.
\end{equation*}
Each term of $\chi_n'''$ has a factor of $(x-\epsilon)^3$ implying the above limit vanishes. Hence
\begin{equation} \label{CutoffNRatio}
\sup_{x\in[\epsilon,b+\epsilon]} \left|\frac{(\chi_n'''(x;\epsilon,b))^2}{\chi_n'(x;\epsilon,b)} \right|
	\leq c(n;b)
\end{equation}
and so
\begin{equation} \label{CutoffNRatio2}
\left| \frac{(\chi_n'''(x;\epsilon,b))^2}{\chi_n'(x;\epsilon,b)} \right|
	\leq c(n;b)(1+\chi_n(x;\epsilon,b)).
\end{equation}
Each term of \eqref{CutoffNDerivative} is nonnegative and $\chi'$ is supported in $[\epsilon,b+\epsilon]$, hence
\begin{equation*}
\chi_n'(x;\epsilon,b) \leq c(n;b)(1+\chi_n(x;\epsilon,b)).
\end{equation*}
Using the Leibniz rule, it similarly follows for $j=1,2,3,4,5$ that
\begin{equation} \label{CutoffNDerivatives}
|\chi_n^{(j)}(x;\epsilon,b)| \leq c(n;j;b)(1+\chi_n(x;\epsilon,b)).
\end{equation}
Assuming $n\geq3$, notice that \eqref{CutoffNRatio} and
\begin{equation*}
\frac{(\chi_n'''(x;\epsilon,b))^2}{\chi_n'(x;\epsilon,b)} = (n-1)(n-2)x^{n-5}
	\qquad (b+\epsilon \leq x)
\end{equation*}
imply
\begin{equation} \label{CutoffNRatioToMinusOne}
\left|\frac{(\chi_n'''(x;\epsilon,b))^2}{\chi_n'(x;\epsilon,b)} \right|
	\leq c(n;\epsilon;b)\chi_{n-1}(x;\epsilon/3,b+\epsilon).
\end{equation}
A similar argument holds for $n=1,2$. Next we prove for $j=1,2,3,4,5$
\begin{equation} \label{CutoffNPrimeToMinusOne}
|\chi_n^{(j)}(x;\epsilon,b)| \leq c(n;j;\epsilon;b) \chi_{n-1}(x;\epsilon/3,b+\epsilon).
\end{equation}
This follows by definition when $b+\epsilon \leq x$; thus it suffices to prove
\begin{equation*}
\sup_{x\in[\epsilon,b+\epsilon]} \left|\frac{\chi_n^{(j)}(x;\epsilon,b)}{\chi_{n-1}(x;\epsilon/3,b+\epsilon)}\right| \leq c(n,j,\epsilon,b).
\end{equation*}
We demonstrate the details for $j=1$, the remaining cases being similar. In this case
\begin{equation*}
\frac{\chi_n^{(j)}(x;\epsilon,b)}{\chi_{n-1}(x;\epsilon/3,b+\epsilon)}
	= \frac{n\chi(x;\epsilon,b)}{\chi(x;\epsilon/3,b+\epsilon)}
		+ \frac{x\chi'(x;\epsilon,b)}{\chi(x;\epsilon/3,b+\epsilon)}.
\end{equation*}
Assuming $\epsilon \leq x \leq b+\epsilon$,
\begin{equation*}
\frac{n\chi(x;\epsilon,b)}{\chi(x;\epsilon/3,b+\epsilon)}
	= n\left(\frac{b+\epsilon}{b}\right)^6 \frac{(x-\epsilon)^6p\left(\frac{x-\epsilon}{b}\right)}{(x-\frac\epsilon3)^6p\left(\frac{x-\frac\epsilon3
}{b+\epsilon}\right)}.
\end{equation*}
Note that $\frac{x-\frac\epsilon3}{b+\epsilon}<1$ so that $p$ does not vanish in $[\epsilon,b+\epsilon]$. Hence this above expression is continuous and bounded on this interval. Similarly for the second term
\begin{equation*}
\frac{x\chi'(x;\epsilon,b)}{\chi(x;\epsilon/3,b+\epsilon)}
	= \frac{2772(b+\epsilon)^6(x-\epsilon)^5(b-x+\epsilon)^5x}{b^{11}(x-\frac\epsilon3)p\left(\frac{x-\frac\epsilon3
}{b+\epsilon}\right)}.
\end{equation*}
This proves \eqref{CutoffNPrimeToMinusOne} in the case $j=1$.
\end{section}

\begin{section}{Proof of Theorem 1}\label{S:3}

In this section, we prove Theorem 1. 
We show several lemmas which are needed to 
prove Theorems \ref{PROP}, \ref{DECAY} and \ref{DIR}. 
The first lemma is an analogue of \eqref{KATO} to implement Kato's energy estimate argument which is proved by 
Isaza-Linares-Ponce \cite{MR3279353}. 

\begin{lemma}
Let $u \in C^{\infty}([0,T];H^\infty(\mathbb{R}))$ be 
a solution to IVP
\begin{equation}
\begin{cases}
	\partial_tu - \partial_x^5u = F
		\qquad x,t\in\mathbb{R} \\
	u(x,0) = u_0(x)
\end{cases}
\end{equation}
and let $\psi \in C^5(\mathbb{R}^2)$ 
satisfy $\partial_x\psi\geq0$. Then we have
\begin{align} \label{ENERGY}
& \frac{d}{dt} \int u^2\psi \; dx + \int (\partial_x^2u)^2\partial_x\psi \; dx \notag \\
	&\qquad\qquad
		\leq \int u^2\left\{\partial_t\psi + \frac32\partial_x^5\psi + \frac{25}{16}\frac{(\partial_x^3\psi)^2}{\partial_x\psi}\right\} \; dx
		+ 2 \int u F \psi \; dx.
\end{align}
\end{lemma}

By interpolation we have the following lemma, which is required to apply the inductive hypothesis.

\begin{lemma} \label{PLemmaInduct}
Suppose $u_0\in L^2(\mathbb{R})$ and for some $l\in\mathbb{Z}^+$, $l\geq2$, $x_0\in\mathbb{R}$
\begin{equation}
\|\partial_x^lu_0\|_{L^2(x_0,\infty)}^2
	= \int_{x_0}^\infty |\partial_x^lu_0|^2 \; dx
	< \infty.
\end{equation}
For any $k=1,2,\dots,l-1$ and $\delta>0$
\begin{equation}
\|\partial_x^ku_0\|_{L^2(x_0+\delta,\infty)}^2
	= \int_{x_0+\delta}^\infty |\partial_x^ku_0|^2 \; dx
	< \infty.
\end{equation}
\end{lemma}

We reproduce for convenience a lemma in the work of Isaza, Linares and Ponce \cite{IsazaLinaresPonce2014a}.
\begin{lemma} 
Let $j_1,j_2,j_3\in\mathbb{Z}^+$ and $\epsilon,b>0$. Suppose $\psi(x;\epsilon,b)$ has support in $[\epsilon,\infty)$, $\psi\geq0$ and $\psi(x;\epsilon,b)\geq1$ whenever $x\geq b+\epsilon$. Then
\begin{align} \label{LinftyTrick}
&\int |\partial_x^{j_1}u\partial_x^{j_2}u\partial_x^{j_3}u|\psi(x) \; dx \notag \\
	&\qquad\qquad
		\lesssim \left\{\int(\partial_x^{1+j_1}u)^2\psi(x) \; dx
		+ \int(\partial_x^{j_1}u)^2\psi(x) \; dx
		+ \int(\partial_x^{j_1}u)^2|\psi'(x)| \; dx\right\} \notag \\
	&\qquad\qquad\qquad
		\times \int(\partial_x^{j_2}u)^2\psi(x;\epsilon/5,4\epsilon/5) \; dx
		+ \int(\partial_x^{j_3}u)^2\psi(x) \; dx.
\end{align}
In particular, we may choose $\psi=\chi, \chi', \chi_n$ or $\chi_n'$.
\end{lemma}
\begin{proof}
Using Cauchy-Schwarz and Young's inequality, followed by the Sobolev embedding, we have
\begin{align*}
\lefteqn{\int |\partial_x^{j_1}u\partial_x^{j_2}u\partial_x^{j_3}u|\psi \; dx}\\
	&\leq
		\frac12\int (\partial_x^{j_1}u)^2(\partial_x^{j_2}u)^2\psi \; dx
		+ \frac12\int(\partial_x^{j_3}u)^2\psi \; dx \\
	&\leq
		\frac12\|(\partial_x^{j_1}u)^2\psi\|_{L_x^\infty}
			\int_\epsilon^\infty(\partial_x^{j_2}u)^2 \; dx
		+ \frac12\int(\partial_x^{j_3}u)^2\psi \; dx \\
	&\leq
		\frac12\|\partial_x((\partial_x^{j_1}u)^2\psi)\|_{L_x^1}
			\int(\partial_x^{j_2}u)^2\psi(x;\epsilon/5,4\epsilon/5) \; dx
		+ \frac12\int(\partial_x^{j_3}u)^2\psi \; dx
\end{align*}
since $\psi(x;\epsilon,b)$ is nonnegative, supported on $[\epsilon,\infty)$ and $\psi(x;\epsilon,b)\geq1$ when $x \geq b+\epsilon$. Furthermore, Young's inequality yields
\begin{align*}
\|\partial_x((\partial_x^{j_1}u)^2\psi)\|_{L_x^1}
	&\leq
		2\int|\partial_x^{j_1}u\partial_x^{1+j_1}u|\psi \; dx
		+ \int(\partial_x^{j_1}u)^2|\psi'| \; dx \\
	&\leq
		\int(\partial_x^{1+j_1}u)^2\psi \; dx
		+ \int(\partial_x^{j_1}u)^2\psi \; dx
		+ \int(\partial_x^{j_1}u)^2|\psi'| \; dx.
\end{align*}
This completes the proof of Lemma 3.
\end{proof}

We now turn to the proof of Theorem \ref{PROP}. As the argument is translation invariant, we consider only $x_0=0$. Additionally, the estimates are performed for nonlinearity $u\partial_x^3u$; a later remark explains how to control other terms. We invoke constants $c_0,c_1,c_2,\dots,$ depending only on the parameters
\begin{equation} \label{PCONST}
c_k=c_k(l,T,\epsilon,b,\|u_0\|_{H^s}; \|\partial_x^lu_0\|_{L^2(x_0,\infty)}; \|\partial_x^3u\|_{L_T^1L_x^\infty})
\end{equation}
whose value may change from line to line. We explicitly record dependence on the parameter $\nu$ using the notation $c(\nu;d)$, which indicates a constant taking the form of a degree-$d$ polynomial in $\nu$:
\begin{equation*}
c(\nu;d) = c_d\nu^d + \cdots + c_1\nu + c_0.
\end{equation*}

We first describe the formal calculations and later provide justification using a limiting argument. Let $u$ be a smooth solution of IVP \eqref{KDV5G}, differentiate the equation $l$-times and apply \eqref{ENERGY} with $\phi(x,t)=\chi(x+\nu t;\epsilon,b)$. Using properties \eqref{CutoffRatioExpanded} and \eqref{CutoffExpanded} to expand the region of integration in the first term, we arrive at
\begin{align}
& \frac{d}{dt} \int (\partial_x^lu)^2\chi(x+\nu t) \; dx
	+ \int (\partial_x^{l+2}u)^2 \chi'(x+\nu t) \; dx \notag \\
	& \qquad\qquad\qquad
		\leq \int (\partial_x^lu)^2\left\{\nu\chi'(x+\nu t) + \frac32\chi^{(5)}(x+\nu t) + \frac{25}{16}\frac{(\chi'''(x+\nu t))^2}{\chi'(x+\nu t)}\right\} \; dx \notag \\
	& \qquad\qquad\qquad\qquad\qquad
		+ 2\int\partial_x^lu\partial_x^l(u\partial_x^3u)\chi(x+\nu t) \; dx \notag \\
	& \qquad\qquad\qquad
		\leq A+B,\label{PINEQ}
\end{align}		
where		
\begin{align*}		
	A&=\nu\int (\partial_x^lu)^2\chi'(x+\nu t) \; dx+c(\epsilon;b)\int (\partial_x^lu)^2\chi'(x+\nu t;\epsilon/3,b+\epsilon) \; dx,\\
	B&=2\int\partial_x^lu\partial_x^l(u\partial_x^3u)\chi(x+\nu t) \; dx. 
\end{align*}
We have used the convention that when $\epsilon$ and $b$ are suppressed, $\chi(x)=\chi(x;\epsilon,b)$. The argument proceeds via induction on $l$ where, for fixed $l$, we integrate \eqref{PINEQ} in time, integrate $B$ by parts and apply a correction to account for the loss of derivatives.
\newline

\underline{Case $l=1$}
Integrating in the time interval $[0,t]$ and applying \eqref{CutoffBounded}, we obtain
\begin{equation} \label{P1A_INEQ}
\left|\int_0^t A \; d\tau\right|
	\leq c_0(1+\nu) \int_0^t \int (\partial_xu)^2 \; dxd\tau
	\leq c_0(1+\nu)T\|u\|_{L_T^\infty H_x^1}^2
\end{equation}
where $0\leq t\leq T$. After integrating by parts, we find
\begin{align}
B
	&= \int \partial_xu(\partial_x^2u)^2 \chi(x+\nu t) \; dx + 3 \int u(\partial_x^2u)^2 \chi'(x+\nu t) \; dx \notag \\
	&\qquad
		+ \frac43 \int (\partial_xu)^3 \chi''(x+\nu t) \; dx
		- \int u(\partial_xu)^2\chi'''(x+\nu t) \; dx.
\end{align}
The inequality \eqref{CutoffBounded} 
and the Sobolev embedding imply 
\begin{align} \label{P1B_INEQ}
\left|\int_0^t B \; d\tau\right|
	&\leq c_1(\|\partial_xu\|_{L_T^\infty L_x^\infty}+\|u\|_{L_T^\infty L_x^\infty}) \int_0^t\int(\partial_xu)^2 + (\partial_x^2u)^2 \; dxd\tau \notag \\
	&\leq c_1T\|u\|_{L_T^\infty H_x^2}^3.
\end{align}
Integrating the inequality \eqref{PINEQ} and combining \eqref{P1A_INEQ} and \eqref{P1B_INEQ}, 
we obtain 
\begin{align*}
&\int (\partial_xu)^2\chi(x+\nu t) \; dx
		+ \int_0^t \int (\partial_x^3u)^2 \chi'(x+\nu \tau) \; dxd\tau \\
	&\qquad\qquad
		\leq \int (\partial_xu_0)^2\chi(x) \; dx +  \left| \int_0^t A +B \; d\tau \right| \\
	&\qquad\qquad
		\leq c_0\nu + c_1.
\end{align*}
As the right-hand side is independent of $t$, the result follows.
\newline

\underline{Case $l=2$}
Similar to the previous case, integrating in the time interval $[0,t]$, we find
\begin{equation} \label{P2A_INEQ}
\left|\int_0^t A \; d\tau\right|
	\leq c_0(1+\nu) \int_0^t \int (\partial_x^2u)^2 \; dxd\tau
	\leq c_0(1+\nu)T\|u\|_{L_T^\infty H_x^2}^2
\end{equation}
where $0\leq t\leq T$. After integrating by parts, 
we see
\begin{align} \label{P2B}
B
	&= - \int \partial_xu(\partial_x^3u)^2 \chi(x+\nu t) \; dx + 3 \int u(\partial_x^3u)^2 \chi'(x+\nu t) \; dx \notag \\
	&\qquad
		- \int \partial_xu(\partial_x^2u)^2\chi''(x+\nu t) \; dx
		- \int u(\partial_x^2u)^2\chi'''(x+\nu t) \; dx.
\end{align}
This expression exhibits a loss of derivatives in that the term
\begin{equation} \label{P2B_LOSS}
\int \partial_xu (\partial_x^3u)^2 \chi(x+\nu t) \; dx
\end{equation}
can be controlled neither by the well-posedness theory nor by the $l=1$ case (without the technique introduced in Section 7). In \cite{MR2455780}, Kwon introduced a modified energy to overcome a similar issue. In particular, a smooth solution $u$ to the IVP \eqref{KDV5G} satisfies the following identity:
\begin{eqnarray} 
\lefteqn{\frac{d}{dt} \int u(\partial_xu)^2\chi \; dx}
\nonumber\\
	&=& -5\int\partial_xu(\partial_x^3u)^2\chi\;dx
		-5\int u(\partial_x^3u)^2\chi'\;dx
		+ \frac{28}{3}\int (\partial_x^2u)^3\chi' \; dx 
		\nonumber\\
	& &+ 21\int\partial_xu(\partial_x^2u)^2\chi'' \; dx
		+ 5\int u(\partial_x^2u)^2\chi''' \; dx
		- \frac{10}{3} \int (\partial_xu)^3\chi^{(4)} \; dx \
		\nonumber \\
	& &- \int u(\partial_xu)^2\chi^{(5)} \; dx
		+ 4\int u\partial_xu(\partial_x^2u)^2\chi \; dx
		+ 3\int u^2(\partial_x^2u)^2\chi' \; dx \nonumber \\
	& &- \frac94\int(\partial_xu)^4 \chi' \; dx
		- \int u\partial_x^2u(\partial_xu)^2\chi' \; dx
		- 4\int u(\partial_xu)^3 \chi'' \; dx \nonumber \\
	& &- \int u^2(\partial_xu)^2 \chi''' \; dx
		+ \nu \int u(\partial_xu)^2 \chi' \; dx
		\label{P2B_KWON}
\end{eqnarray}
where $\chi^{(j)}$ denotes $\chi^{(j)}(x+\nu t)$. We use this identity to eliminate \eqref{P2B_LOSS} from \eqref{P2B}, yielding
\begin{align} \label{P2BK}
B
	&= \frac15 \frac{d}{dt}\int u(\partial_xu)^2\chi(x+\nu t) \; dx
		+ 4 \int u(\partial_x^3u)^2\chi'(x+\nu t) \; dx \notag \\
	&\qquad
		- \frac45 \int u\partial_xu(\partial_x^2u)^2\chi(x+\nu t) \; dx
		- \frac\nu5 \int u(\partial_xu)^2 \chi'(x+\nu t) \; dx \notag \\
	&\qquad
		+ \sum_{\substack{0 \leq j_1,j_2,j_3 \leq 2 \\ 1 \leq j_4 \leq 5}}c_{j_1,j_2,j_3,j_{4}}
			\int \widetilde{\partial_x^{j_1}u}\partial_x^{j_2}u(\partial_x^{j_3}u)^2\chi^{(j_4)}(x+\nu t) \; dx
\end{align}
where the notation $\widetilde{\partial_x^{j_1}u}$ indicates this factor may be omitted. That is, since $0\leq j_1,j_2 \leq 2$,
\begin{equation*}
\|\widetilde{\partial_x^{j_1}u}\partial_x^{j_2}u\|_{L_T^\infty L_x^\infty}
	\leq \|u\|_{L_T^\infty H_x^s} + \|u\|_{L_T^\infty H_x^s}^2.
\end{equation*}
Integrating in the time interval $[0,t]$, applying \eqref{CutoffBounded} and the Sobolev embedding, 
we obtain
\begin{align} \label{P2B_LOW}
&\left|\int_0^t\int \widetilde{\partial_x^{j_1}u}\partial_x^{j_2}u(\partial_x^{j_3}u)^2\chi^{(j_4)}(x+\nu\tau) \; dxd\tau\right| \notag \\
	&\qquad\qquad\qquad\qquad
		\leq c_1\|\widetilde{\partial_x^{j_1}u}\partial_x^{j_2}u\|_{L_T^\infty L_x^\infty} \int_0^T\int(\partial_x^{j_3}u)^2 \; dxd\tau \notag \\
	&\qquad\qquad\qquad\qquad
		\leq c_1T\|u\|_{L_T^\infty H_x^s}^3(1+\|u\|_{L_T^\infty H_x^s})
\end{align}
since $\max\{j_1,j_2,j_3\}\leq2$. 
The fundamental theorem of calculus and Sobolev embedding 
yield 
\begin{align} \label{P2B_INEQ0}
\left|\int_0^t B \; d\tau\right|
	&\leq \left|\int u_0(\partial_xu_0)^2 \chi(x)\;dx\right|
		+ \left|\int u(\partial_xu)^2\chi(x+\nu t)\; dx\right| \notag \\
	&\qquad
		+ 4 \|u\|_{L_T^\infty H_x^1}\int_0^T\int (\partial_x^3u)^2\chi'(x+\nu\tau) \; dxd\tau \notag \\
	&\qquad
		+ \frac45 \|u\|_{L_T^\infty H_x^2}^2 \int_0^T\int(\partial_x^2u)^2\chi(x+\nu\tau) \; dxd\tau \notag \\
	&\qquad
		+ \frac\nu5 \|u\|_{L_T^\infty H_x^1}\int_0^T\int (\partial_xu)^2 \chi'(x+\nu\tau) \; dxd\tau \notag \\
	&\qquad
		+ c_1T\|u\|_{L_T^\infty H_x^s}^3(1+\|u\|_{L_T^\infty H_x^s}).
\end{align}
The first term on the right-hand side is controlled by the Sobolev embedding, the hypothesis on the initial data and Lemma \ref{PLemmaInduct}. The second and third term illustrate the iterative nature of the argument, as they can be bounded by the $l=1$ result. The two remaining integrals are finite by property \eqref{CutoffBounded}. Therefore
\begin{equation} \label{P2B_INEQ}
\left|\int_0^t B \; d\tau\right|
	\leq c_0\nu+c_1.
\end{equation}
Integrating inequality \eqref{PINEQ}, using \eqref{P2A_INEQ}, \eqref{P2B_INEQ} and the hypothesis on the initial data, 
we have
\begin{align*}
&\int (\partial_x^2u)^2\chi(x+\nu t) \; dx
		+ \int_0^t \int (\partial_x^4u)^2 \chi'(x+\nu \tau) \; dxd\tau \\
	&\qquad\qquad
		\leq \int (\partial_x^2u_0)^2\chi(x) \; dx +  \left|\int_0^t A + B \; d\tau\right| \\
	&\qquad\qquad
		\leq c_0\nu+c_1.
\end{align*}
As the right-hand side is independent of $t$, the result follows.
\newline

\underline{Case $l=3$}
Integrating in the time interval $[0,t]$ and applying the $l=1$ result, we obtain 
\begin{align} \label{P3A_INEQ}
\left|\int_0^t A \; d\tau\right|
	&\leq \nu\int_0^T \int (\partial_x^3u)^2\chi'(x+\nu\tau) \; dxd\tau\notag\\
    &\qquad+ c_0\int_0^T \int (\partial_x^3u)^2\chi'(x+\nu\tau;\epsilon/3,b+\epsilon) \; dxd\tau \notag \\
	&\leq c_2\nu^2 + c_1\nu + c_0
\end{align}
where $0\leq t\leq T$. After integrating by parts, 
we find 
\begin{align} \label{P3B}
B
	&= -3 \int \partial_xu(\partial_x^4u)^2 \chi(x+\nu t) \; dx + 3 \int u(\partial_x^4u)^2 \chi'(x+\nu t) \; dx \notag \\
	&\qquad
		+ \int (\partial_x^3u)^3\chi(x+\nu t) \; dx
		- \int u(\partial_x^3u)^2\chi'''(x+\nu t) \; dx.
\end{align}
This expression exhibits a loss of derivatives in the term
\begin{equation} \label{P3B_LOSS}
\int \partial_xu (\partial_x^4u)^2 \chi(x+\nu t) \; dx.
\end{equation}
A smooth solution $u$ to the IVP \eqref{KDV5G} satisfies the following identity:
\begin{align} \label{P3B_KWON}
\lefteqn{\frac{d}{dt} \int u(\partial_x^2u)^2\chi \; dx}
\notag\\
	&= -5\int\partial_xu(\partial_x^4u)^2\chi\;dx
		- 5\int u(\partial_x^4u)^2\chi'\;dx\notag \\
	&\qquad
		+ 5\int (\partial_x^3u)^3\chi \; dx 
		+ 25\int\partial_x^2u(\partial_x^3u)^2\chi' \; dx
		+ 15\int \partial_xu(\partial_x^3u)^2\chi'' \; dx \notag \\
	&\qquad
		+ 5\int u(\partial_x^3u)^2\chi''' \; dx
		+ 2\int u\partial_xu(\partial_x^3u)^2\chi \; dx
		+ 3\int u^2(\partial_x^3u)^2\chi' \; dx \notag \\
	&\qquad
		- \frac{25}{3}\int (\partial_x^2u)^3\chi''' \; dx
		- 5\int \partial_xu(\partial_x^2u)^2\chi^{(4)} \; dx
		- \int u(\partial_x^2u)^2\chi^{(5)} \; dx \notag \\
	&\qquad
		- \int\partial_xu(\partial_x^2u)^3 \chi \; dx
		- 3\int u(\partial_x^2u)^2 \chi' \; dx 
		- 2\int(\partial_xu)^2(\partial_x^2u)^2\chi' \; dx\notag \\
	&\qquad
		- 4\int u\partial_xu(\partial_x^2u0^2\chi'' \; dx
		- \int u^2(\partial_x^2u)^2 \chi''' \; dx 
		+ \nu \int u(\partial_x^2u)^2\chi' \; dx
\end{align}
where $\chi^{(j)}$ denotes $\chi^{(j)}(x+\nu t)$, which we use to eliminate \eqref{P3B_LOSS} from \eqref{P3B}. Thus, ignoring coefficients, we may write
\begin{align} \label{P3BK}
B
	&= \frac{d}{dt}\int u(\partial_x^2u)^2\chi(x+\nu t) \; dx
		+ \int u(\partial_x^4u)^2\chi'(x+\nu t) \; dx \notag \\
	&\qquad
		+ \int(1+u\partial_xu+\partial_x^3u)(\partial_x^3u)^2\chi(x+\nu t) \; dx
		+ \nu \int u(\partial_x^2u)^2\chi' \; dx \notag \\
	&\qquad
		+ \sum_{\substack{0\leq j_1,j_2\leq2 \\ 1 \leq j_3 \leq 3}}c_{j_1,j_2,j_{3}}
			\int \widetilde{\partial_x^{j_1}u}\partial_x^{j_2}u(\partial_x^3u)^2\chi^{(j_3)}(x+\nu t) \; dx \notag \\
	&\qquad
		+ \sum_{\substack{0\leq j_1,j_2\leq2 \\ 1 \leq j_3 \leq 5}}c_{j_1,j_2,j_{3}}
			\int \widetilde{\partial_x^{j_1}u}\partial_x^{j_2}u(\partial_x^2u)^2\chi^{(j_3)}(x+\nu t) \; dx
\end{align}
where the notation $\widetilde{\partial_x^{j_1}u}$ indicates this factor may be omitted. Integrating in the time interval $[0,t]$, applying \eqref{CutoffExpanded}, the Sobolev embedding and the $l=1$ result yields
\begin{align} \label{P3B_LOW1}
\lefteqn{\left|\int_0^t\int \widetilde{\partial_x^{j_1}u}\partial_x^{j_2}u(\partial_x^3u)^2\chi^{(j_3)}(x+\nu\tau) \; dxd\tau\right|} \notag \\
	&\leq c_1\|\widetilde{\partial_x^{j_1}u}\partial_x^{j_2}u\|_{L_T^\infty L_x^\infty} \int_0^T\int(\partial_x^3u)^2\chi'(x+\nu\tau;\epsilon/3,b+\epsilon) \; dxd\tau \notag \\
	&
		\leq (\|u\|_{L_T^\infty H_x^s}+\|u\|_{L_T^\infty H_x^s}^2)(c_0\nu+c_1).
\end{align}
Similarly, integrating in the time interval $[0,t]$, applying \eqref{CutoffBounded} and the Sobolev embedding, we find 
\begin{align} \label{P3B_LOW2}
\lefteqn{\left|\int_0^t\int \widetilde{\partial_x^{j_1}u}\partial_x^{j_2}u(\partial_x^2u)^2\chi^{(j_3)}(x+\nu\tau) \; dxd\tau\right| }\notag \\
	&
		\leq c_1\|\widetilde{\partial_x^{j_1}u}\partial_x^{j_2}u\|_{L_T^\infty L_x^\infty} \int_0^T\int(\partial_x^2u)^2 \; dxd\tau \notag \\
	&
		\leq c_1T\|u\|_{L_T^\infty H_x^s}^3(1+\|u\|_{L_T^\infty H_x^s}).
\end{align}
Hence the fundamental theorem of calculus and Sobolev embedding yield 
\begin{align} \label{P3B_INEQ0}
\left|\int_0^t B \; d\tau\right|
	&\leq \left|\int u_0(\partial_x^2u_0)^2 \chi(x)\;dx\right|
		+ \left|\int u(\partial_x^2u)^2\chi(x+\nu t)\; dx\right| \notag \\
	&\qquad
		+ \|u\|_{L_T^\infty H_x^1}\int_0^T\int(\partial_x^4u)^2\chi'(x+\nu\tau) \; dxd\tau \notag \\
	&\qquad
		+ \int_0^t (1+\|u\|_{L_T^\infty H_x^2}^2 + \|\partial_x^3u(\tau)\|_{L_x^\infty}) \int(\partial_x^3u)^2\chi(x+\nu\tau) \; dxd\tau \notag \\
	&\qquad
		+ (\|u\|_{L_T^\infty H_x^s}+\|u\|_{L_T^\infty H_x^s}^2)(c_0\nu+c_1) \notag \\
	&\qquad
		+ c_1T\|u\|_{L_T^\infty H_x^s}^3(1+\|u\|_{L_T^\infty H_x^s}).
\end{align}
Similar to the $l=2$ case, the first term on the right-hand side is controlled by the hypothesis on the initial data. The second and third terms are finite by the $l=2$ case. Therefore
\begin{equation} \label{P3B_INEQ}
\left|\int_0^t B \; d\tau\right|
	\leq c(\nu;1) + \int_0^t (c_0+c_1\|\partial_x^3u(\tau)\|_{L_x^\infty}) \int(\partial_x^3u)^3\chi(x+\nu\tau) \; dxd\tau.
\end{equation}
Integrating inequality \eqref{PINEQ}, using \eqref{P3A_INEQ}, \eqref{P3B_INEQ} and the hypothesis on the initial data, 
we have
\begin{align*}
y(t)
	&:= \int (\partial_x^3u)^2\chi(x+\nu t) \; dx
		+ \int_0^t \int (\partial_x^5u)^2 \chi'(x+\nu \tau) \; dxd\tau \\
	&\leq \int (\partial_x^3u_0)^2\chi(x) \; dx +  \left|\int_0^t A + B \; d\tau\right| \\
	&\leq c(\nu;2) + \int_0^t (c_0+c_1\|\partial_x^3u(\tau)\|_{L_x^\infty}) \int(\partial_x^3u)^2\chi(x+\nu\tau) \; dxd\tau \\
	&\leq c(\nu;2) + \int_0^t (c_0+c_1\|\partial_x^3u(\tau)\|_{L_x^\infty}) y(\tau) \; dxd\tau.
\end{align*}
Applying Gronwall's inequality produces
\begin{eqnarray*}
\lefteqn{\sup_{0\leq t \leq T} \int (\partial_x^4u)^2\chi(x+\nu t) \; dx
	+ \int_0^T\int(\partial_x^5u)^2 \chi'(x+\nu\tau) \; dxd\tau}\\
	&\leq& c(\nu;2)\exp\left(c_0T + c_1\|\partial_x^3u\|_{L_T^1L_x^\infty}\right).\qquad\qquad
\end{eqnarray*}
This proves the desired result with $l=3$. 
\newline

\underline{Cases $l=4,5,6$} Due to the structure of the IVP, the cases $l=4,5,6$ must be handled individually. The analysis is omitted as it is similar to the cases $l=3$ and $l\geq7$. It can be proved that
\begin{equation*}
\sup_{0\leq t \leq T} \int (\partial_x^lu)^2\chi(x+\nu t) \; dx
	+ \int_0^T \int (\partial_x^{l+2}u)^2 \chi'(x+\nu \tau) \; dxd\tau
		\leq c(\nu;d)
\end{equation*}
where the values of $d$ are summarized in the following table.
\begin{table}[h]
\centering
\begin{tabular}{l|llllll}
	$l$ & 1 & 2 & 3 & 4 & 5 & 6 \\
	$d$ & 1 & 1 & 2 & 2 & 4 & 8
\end{tabular}
\end{table}

\underline{Case $l\geq7$} In the course of this case, we will prove that for $l\geq7$, the final constant obtained after integrating both sides of \eqref{PINEQ} takes the form of a polynomial in $\nu$ with degree $8(l-5)$.

Integrating in the time interval $[0,t]$ and applying the $l-2$ result (assuming $l>7$) we have
\begin{align} \label{PLA_INEQ}
\left|\int_0^t A \; d\tau\right|
	&\leq \nu\int_0^T \int (\partial_x^lu)^2\chi'(x+\nu\tau) \; dxd\tau\notag\\
	&\qquad+ c_0\int_0^T \int (\partial_x^lu)^2\chi'(x+\nu\tau;\epsilon/3,b+\epsilon) \; dxd\tau \notag \\
	&\leq c(\nu;1+8(l-7))
\end{align}
where $0\leq t\leq T$. For $l=7$, this expression has degree 5 in $\nu$. We write
\begin{align} \label{PLB}
B=B_{1}+B_{2}
\end{align}
where 
\begin{align*}
B_{1}&=2\int \partial_x^lu
		\left\{u\partial_x^{l+3}u
		+\binom{l}{1}\partial_xu\partial_x^{l+2}u
		+\binom{l}{2}\partial_x^2u\partial_x^{l+1}u
		\right.\notag\\
	&\qquad\qquad\qquad\qquad\qquad\left.+(1+\binom{l}{3})\partial_x^3u\partial_x^lu\right\}\chi(x+\nu t) \; dx\\
B_{2}&=\sum_{k=1}^{\ceil{l/2}-2} c_{l,k}\int\partial_x^{3+k}u\partial_x^{l-k}u\partial_x^lu\chi(x+\nu t) \; dx
\end{align*}
and $3+k\leq l-k<l$ for $1\leq k\leq\ceil{l/2}-2$. Integrating by parts, we have
\begin{align} \label{PLB1}
B_1=B_{11}+B_{12},
\end{align}
where
\begin{align*}
B_{11}&=(3-2l)\int\partial_xu(\partial_x^{l+1}u)^2\chi(x+\nu t) \; dx,\\
B_{12}&=
\int u(\partial_x^{l+1}u)^2\chi'(x+\nu t) \; dx
+\int\partial_x^3u(\partial_x^lu)^2\chi(x+\nu t) \; dx\\
&\qquad+ \int\partial_x^2u(\partial_x^lu)^2\chi'(x+\nu t) \; dx
+ \int\partial_xu(\partial_x^lu)^2\chi''(x+\nu t) \; dx\\
&\qquad+ \int u(\partial_x^lu)^2\chi'''(x+\nu t) \; dx
\end{align*}
and, in $B_{12}$, we have omitted coefficients depending only on $l$ using the expression \eqref{PLB1}. 
Then integrating in the time interval $[0,t]$, 
where $0 \leq t \leq T$, we obtain 
\begin{align*}
\left|\int_0^t B_{12} \; d\tau\right|
	&\leq \|u\|_{L_T^\infty H_x^1}\int_0^T\int (\partial_x^{l+1}u)^2\chi'(x+\nu\tau) \; dxd\tau \notag \\
	&\qquad
		+ \int_0^t \|\partial_x^3u(\tau)\|_{L_x^\infty} \int(\partial_x^lu)^2\chi(x+\nu\tau) \; dxd\tau \notag \\
	&\qquad
		+ c_0\|u\|_{L_T^\infty H_x^s}\int_0^t\int(\partial_x^lu)^2\chi'(x+\nu\tau) \; dxd\tau
\end{align*}
by the Sobolev embedding and \eqref{CutoffExpanded}. Applying the result for cases $l-1$ and $l-2$, 
we have 
\begin{equation} \label{PLB12_INEQ}
\left|\int_0^t B_{12} \; d\tau\right|
	\leq c(\nu;8(l-6))
		+ \int_0^t \|\partial_x^3u(\tau)\|_{L_x^\infty} \int(\partial_x^lu)^2\chi(x+\nu\tau) \; dxd\tau.
\end{equation}
Observe that term $B_2$ only occurs when $l\geq5$. For $l>5$, note that $4+k<l$. The inequality \eqref{LinftyTrick} 
produces 
\begin{align} \label{PLB2}
\left|B_2\right|
	&\leq \sum_{k=1}^{\ceil{l/2}-2} c_{l,k} \int|\partial_x^{3+k}u\partial_x^{l-k}u\partial_x^lu|\chi(x+\nu t) \; dx \notag \\
	&\leq \int(\partial_x^lu)^2\chi(x+\nu t) \; dx \notag \\
	&\qquad
		+ \sum_{k=1}^{\ceil{l/2}-2} \left\{\int(\partial_x^{4+k}u)^2\chi(x+\nu t) \; dx
		+ \int(\partial_x^{3+k}u)^2\chi(x+\nu t) \; dx
		\right.
		 \notag \\
	&\qquad\qquad\qquad\left.
	+ \int(\partial_x^{3+k}u)^2\chi'(x+\nu t) \; dx\right\}
	\int(\partial_x^{l-k}u)^2\chi(x+\nu t;\epsilon/5,4\epsilon/5) \; dx,
\end{align}
after suppressing constants depending on $l$. Integrating in the time interval $[0,t]$, we have
\begin{align}
&\left| \int_0^t B_2 \; d\tau\right|\notag \\
	&\qquad
		\leq \int_0^t\int(\partial_x^lu)^2\chi(x+\nu\tau) \; dx \notag \\
	&\qquad\qquad
		+ T\sum_{k=1}^{\ceil{l/2}-2} \left(\sup_{0 \leq t \leq T}\int(\partial_x^{l-k}u)^2\chi(x+\nu t;\epsilon/5,4\epsilon/5) \; dx\right)\notag \\
	&\qquad\qquad\qquad\qquad\qquad\qquad\qquad\times		
			\left(\sup_{0 \leq t \leq T}\int(\partial_x^{4+k}u)^2\chi(x+\nu t) \; dx\right) \notag \\
	&\qquad\qquad
		+ T\sum_{k=1}^{\ceil{l/2}-2} \left(\sup_{0 \leq t \leq T}\int(\partial_x^{l-k}u)^2\chi(x+\nu t;\epsilon/5,4\epsilon/5) \; dx\right)\notag \\
	&\qquad\qquad\qquad\qquad\qquad\qquad\qquad\times		
			\left(\sup_{0 \leq t \leq T}\int(\partial_x^{3+k}u)^2\chi(x+\nu t) \; dx\right) \notag \\
	&\qquad\qquad
		+ T\sum_{k=1}^{\ceil{l/2}-2} \left(\sup_{0 \leq t \leq T}\int(\partial_x^{l-k}u)^2\chi(x+\nu t;\epsilon/5,4\epsilon/5) \; dx\right)\notag \\
	&\qquad\qquad\qquad\qquad\qquad\qquad\qquad\times\left(\sup_{0 \leq t \leq T}\int(\partial_x^{3+k}u)^2\chi'(x+\nu t) \; dx\right).\notag
\end{align}
The strongest $\nu$-dependence for $B_2$ arises from analyzing terms of the form:
\begin{equation} \label{PLB2_NU}
\left(\sup_{0 \leq t \leq T} \int (\partial_x^{l-k}u)^2\chi(x+\nu t;\epsilon/5,4\epsilon/5) \; dx \right)
		\left(\sup_{0 \leq t \leq T} \int(\partial_x^{4+k}u)^2\chi(x+\nu t) \; dx\right).
\end{equation}
Each factor in \eqref{PLB2_NU} is finite by the result for cases $l-k$ and $4+k$. The inductive hypothesis further implies that the $\nu$-dependence has the form of a polynomial in $\nu$ having degree
\begin{equation*}
\nu^{8(l-k-5)}\cdot\nu^{8(4+k-5)} = \nu^{8(l-6)}.
\end{equation*}
Hence
\begin{equation} \label{PLB2_INEQ}
\left| \int_0^t B_2 \; d\tau \right|
	\leq c(\nu;8(l-6)) + c_0\int_0^t\int(\partial_x^lu)^2\chi(x+\nu\tau) \; dxd\tau.
\end{equation}
Integrating the inequality \eqref{PINEQ} in the time interval $[0,t]$, where $0 \leq t \leq T$, 
we have
\begin{align} \label{PL_INEQ0}
&\int (\partial_x^lu)^2\chi(x+\nu t) \; dx
	+ \int_0^t\int (\partial_x^{l+2}u)^2 \chi'(x+\nu\tau) \; dxd\tau \notag \\
	&\qquad\qquad\leq
		\int (\partial_x^lu_0)^2\chi(x) \; dx
			+ \left|\int_0^t A + B_{11} + B_{12} + B_2 \; d\tau\right| \notag \\
	&\qquad\qquad\leq
		c(\nu;8(l-6))\notag \\
	&\qquad\qquad\qquad
	 + \left|\int_0^t B_{11} \; d\tau\right|
		+ \int_0^t (c_0+c_1\|\partial_x^3u(\tau)\|_{L_x^\infty})\int(\partial_x^lu)^2\chi(x+\nu\tau) \; dxd\tau
\end{align}
using the hypothesis on the initial data, \eqref{PLA_INEQ}, \eqref{PLB12_INEQ} and \eqref{PLB2_INEQ}. Thus it only remains to estimate the integral involving
\begin{equation*}
B_{11} = (3-2l)\int\partial_xu(\partial_x^{l+1}u)^2\chi(x+\nu t) \; dx,
\end{equation*}
which exhibits a loss of derivatives. Assuming that $u$ satisfies the IVP \eqref{KDV5G}, we rewrite this term by considering the correction factor
\begin{eqnarray}
\lefteqn{\frac{d}{dt} \int u(\partial_x^{l-1}u)^2\chi(x+\nu t) \; dx} \notag \\
	&=&\int\partial_x^5u(\partial_x^{l-1}u)^2\chi(x+\nu t) \; dx
		+\int u\partial_x^3u(\partial_x^{l-1}u)^2\chi(x+\nu t) \; dx \notag \\
	& &+ 2\int u\partial_x^{l-1}u\partial_x^{l+4}u\chi(x+\nu t) \; dx
		+2\int u\partial_x^{l-1}u\partial_x^{l-1}(u\partial_x^3u)\chi(x+\nu t) \; dx \notag \\
	& &+ \nu \int u(\partial_x^{l-1}u)^2\chi'(x+\nu t) \; dx
		\notag\\
	&=:&C_{1}+C_{2}+\widetilde{C_3}+C_{4}+C_{5}.
\end{eqnarray}
Observe that integrating $\widetilde{C_3}$ by parts reveals
\begin{equation}
\widetilde{C_3} =\left(\frac{5}{2l-3}\right)B_{11}+C_3,
\end{equation}
where
\begin{align} \label{PLC3}
C_3
	&= - 5\int u(\partial_x^{l+1}u)^2\chi' \; dx
		+ 5\int\partial_x^3u(\partial_x^lu)^2\chi \; dx \notag \\
	&\qquad
		+ 9\int\partial_x^2u(\partial_x^lu)^2\chi' \; dx
		+ 15\int\partial_xu(\partial_x^lu)^2\chi'' \; dx
		+ \int u(\partial_x^lu)^2\chi''' \; dx \notag \\
	&\qquad
		- 5\int\partial_x^5u(\partial_x^{l-1}u)^2\chi \; dx
		- 5\int\partial_x^4u(\partial_x^{l-1}u)^2\chi' \; dx
		- 9\int\partial_x^3u(\partial_x^{l-1}u)^2\chi'' \; dx \notag \\
	&\qquad
		- 10\int\partial_x^2u(\partial_x^{l-1}u)^2\chi''' \; dx
		- 5\int\partial_xu(\partial_x^{l-1}u)^2\chi^{(4)} \; dx
		- \int u(\partial_x^{l-1}u)^2\chi^{(5)} \; dx.
\end{align}
Here $\chi^{(j)}$ denotes $\chi^{(j)}(x+\nu t;\epsilon,b)$. 
The fundamental theorem of calculus leads to 
\begin{align} \label{PLB11_INEQ}
\left(\frac{5}{2l-3}\right)\left|\int_0^t B_{11} \; d\tau\right|
	&\leq \left|\int u_0(\partial_x^{l-1}u_0)^2\chi(x) \; dx\right| 
		+ \left|\int u(\partial_x^{l-1}u)^2\chi(x+\nu t) \; dx\right| \notag \\
	&\qquad
		+ \left|\int_0^t C_1 + C_2 + C_3 + C_4 + C_5 \; d\tau\right|.
\end{align}
We now concern ourselves with estimating the right-hand side of this expression. By the Sobolev embedding, hypothesis on the initial data, Lemma \ref{PLemmaInduct} and the result for case $l-1$, we have
\begin{align}
&\left|\int u_0(\partial_x^{l-1}u_0)^2\chi(x) \; dx\right| 
		+ \left|\int u(\partial_x^{l-1}u)^2\chi(x+\nu t) \; dx\right| \notag \\
	&\qquad\qquad\leq
		\|u_0\|_{H^s} \|\partial_x^{l-1}u_0\|_{L_x^2((0,\infty))}^2
		+ \|u\|_{L_T^\infty H_x^s} \int (\partial_x^{l-1}u)^2\chi(x+\nu t) \; dx,
\end{align}
which is uniformly bounded by the inductive hypothesis. Applying \eqref{LinftyTrick}, we obtain
\begin{align*}
|C_1|
	&\leq \int\partial_x^5u(\partial_x^{l-1}u)^2\chi(x+\nu t) \; dx \notag \\
	&\leq \int(\partial_x^{l-1}u)^2\chi(x+\nu t) \; dx \notag \\
	&\qquad
		+ \left\{\int(\partial_x^6u)^2\chi(x+\nu t) \; dx
		+ \int(\partial_x^5u)^2\chi(x+\nu t) \; dx
		+ \int(\partial_x^5u)^2\chi'(x+\nu t) \; dx\right\} \notag \\
	&\qquad\qquad
		\times\int(\partial_x^{l-1}u)^2\chi(x+\nu t;\epsilon/5,4\epsilon/5) \; dx.
\end{align*}
Integrating in the time interval $[0,t]$ and following the argument applied to term $B_2$, we see that the strongest $\nu$-dependence for $C_1$ arises from analyzing the term
\begin{equation} \label{PLC1_NU}
\left(\sup_{0 \leq t \leq T} \int (\partial_x^{l-1}u)^2\chi(x+\nu t;\epsilon/5,4\epsilon/5) \; dx \right)
		\left(\sup_{0 \leq t \leq T} \int(\partial_x^6u)^2\chi(x+\nu t) \; dx\right).
\end{equation}
Each factor in \eqref{PLC1_NU} is finite by the result for cases $6$ and $l-1$. Hence for the base case $l=7$, the right-hand side is bounded by $c(\nu;16)$. For $l>7$, the inductive hypothesis further yields that the $\nu$-dependence has the form of a polynomial in $\nu$ with degree determined by
\begin{equation*}
\nu^{8(l-6)}\cdot\nu^8 = \nu^{8(l-5)}.
\end{equation*}
Thus
\begin{equation} \label{PLC1_INEQ}
\left|\int_0^t C_1 \; d\tau\right| \leq c(\nu;8(l-5)).
\end{equation}
It will be clear from the remainder of the argument that \eqref{PLC1_NU} produces the {\em overall} strongest $\nu$-dependence, hence justifying this inductive calculation.

Integrating in time, using the Sobolev embedding and inductive hypothesis, we find 
\begin{align} \label{PLC2_INEQ}
\left|\int_0^t C_2 \; d\tau\right|
	&\leq
		\|u\|_{L_T^\infty H_x^s}\int_0^T\int |\partial_x^3u|(\partial_x^{l-1}u)^2\chi(x+\nu\tau) \; dxd\tau \notag \\
	&\leq
		\|u\|_{L_T^\infty H_x^s} \int_0^T \|\partial_x^3u(\tau)\|_{L_x^\infty} \left(\sup_{0 \leq t \leq T}\int(\partial_x^{l-1}u)^2\chi(x+\nu t) \; dx\right)d\tau \notag \\
	&\leq
		c(\nu;8(l-6))\|u\|_{L_T^\infty H_x^s} \|\partial_x^3u\|_{L_T^1L_x^\infty}.
\end{align}
Integrating in time and using \eqref{CutoffExpanded}, \eqref{LinftyTrick}, the Sobolev embedding and the inductive hypothesis, we have
\begin{equation} \label{PLC3_INEQ}
\left|\int_0^t C_3 \; d\tau\right|
	\leq c(\nu,8(l-6)) + \int_0^t(c_0+c_1\|\partial_x^3u(\tau)\|_{L_x^\infty})\int(\partial_x^lu)^2\chi(x+\nu\tau) \; dxd\tau.
\end{equation}
Expanding but ignoring binomial coeffiecients, we write $C_4=C_{41}+C_{42}$ with
\begin{align} \label{PLC41}
C_{41}
	&= \int u\partial_xu(\partial_x^lu)^2\chi(x+\nu t) \; dx
		+ \int u^2(\partial_x^lu)^2\chi(x+\nu t) \; dx \notag \\
	&\qquad
		+ \int u\partial_x^3u(\partial_x^{l-1}u)^2\chi(x+\nu t) \; dx
		+ \int\partial_xu\partial_x^2u(\partial_x^{l-1}u)^2\chi(x+\nu t) \; dx \notag \\
	&\qquad
		+ \int u\partial_x^2u(\partial_x^{l-1}u)^2\chi'(x+\nu t) \; dx
		+ \int\partial_xu\partial_xu(\partial_x^{l-1}u)^2\chi'(x+\nu t) \; dx \notag \\
	&\qquad
		+ \int u\partial_xu(\partial_x^{l-1}u)^2\chi''(x+\nu t) \; dx
		- \int u^2(\partial_x^{l-1}u)^2\chi'''(x+\nu t) \; dx
\end{align}
and
\begin{equation} \label{PLC42}
C_{42}
	= \sum_{k=1}^{\lfloor (l-1)/2\rfloor-2} c_{l,k}
		\int u\partial_x^{(l-1)-k}u\partial_x^{3+k}u\partial_x^{l-1}u\chi(x+\nu t) \; dx.
\end{equation}
Similar to $C_2$ and $C_3$,
\begin{equation} \label{PLC41_INEQ}
\left|\int_0^t C_{41} \; d\tau\right|
	\leq c(\nu;8(l-6)) + c_0\int_0^t\int(\partial_x^lu)^2\chi(x+\nu\tau)\;dxd\tau.
\end{equation}
Similar to $B_2$, ignoring constants 
we have
\begin{equation} \label{PLC42_INEQ}
\left|\int_0^t C_{42} \; d\tau\right|
	\leq \sum_{k=1}^{\lfloor (l-1)/2\rfloor-2}
		\int_0^T\int |u\partial_x^{(l-1)-k}u\partial_x^{3+k}u\partial_x^{l-1}u|\chi \; dxd\tau
	\leq c(\nu;8(l-6))
\end{equation}
after applying \eqref{LinftyTrick}. Finally, assuming $l>7$, 
we obtain
\begin{equation*} \label{PLC5_INEQ}
\left|\int_0^t C_5 \; d\tau\right|
	\leq \nu\|u\|_{L_T^\infty H_x^{5/2^+}} \int_0^T\int (\partial_x^{l-1}u)^2\chi'(x+\nu \tau) \; dxd\tau \\
	\leq c(\nu;1+8(l-7))
\end{equation*}
(or $c(\nu;3)$ when $l=7$) using the Sobolev embedding and inductive case $l-3$.

Inserting the above into \eqref{PLB11_INEQ} and \eqref{PL_INEQ0}, then using nonnegativity of $\chi,\chi'$, we find
\begin{align}
y(t)
	&:= \int (\partial_x^lu)^2\chi(x+\nu t) \; dx
		+ \int_0^t\int (\partial_x^{l+2}u)^2 \chi'(x+\nu\tau) \; dxd\tau \notag \\
	&\leq c(\nu;8(l-5)) + \int_0^t(c_0+c_1\|\partial_x^3u(\tau)\|_{L_x^\infty})\int(\partial_x^lu)^2\chi(x+\nu\tau) \; dxd\tau \notag \\
	&\leq c(\nu;8(l-5)) + \int_0^t(c_0+c_1\|\partial_x^3u(\tau)\|_{L_x^\infty})y(\tau)\;d\tau.
\end{align}
Hence Gronwall's inequality yields 
\begin{align*}
&\sup_{0\leq t \leq T} \int (\partial_x^lu)^2\chi(x+\nu t) \; dx
	+ \int_0^T\int(\partial_x^{l+2}u)^2 \chi'(x+\nu\tau) \; dxd\tau \\
	&\qquad\qquad\qquad
		\leq c(\nu;8(l-5))\exp\left(c_0T + c_1\|\partial_x^3u\|_{L_T^1L_x^\infty}\right).
\end{align*}
This concludes the proof for the case of smooth data. 

Now we use a limiting argument to justify the previous computations for arbitrary $u_0 \in H^{s}(\mathbb{R})$ with $s>5/2$. 
Fix $\rho \in C_0^\infty(\mathbb{R})$ with $\text{supp}\;\rho \subseteq (-1,1)$, $\rho\geq0$, $\int\rho(x)\;dx=1$ and
\begin{equation*}
\rho_\mu(x) = \frac1\mu\rho\left(\frac{x}{\mu}\right),	\quad	\mu>0.
\end{equation*}
The the solution $u^\mu$ of IVP \eqref{KDV5G} corresponding to smoothed data $u_0^\mu = \rho_\mu \ast u_0$, $\mu\geq0$, satisfies
\begin{equation*}
u^\mu \in C^{\infty}([0,T] : H^\infty(\mathbb{R})).
\end{equation*}
Hence we may conclude
\begin{equation*}
\sup_{0\leq t \leq T} \int (\partial_x^lu^\mu)^2\chi(x+\nu t) \; dx
	+ \int_0^T\int(\partial_x^{l+2}u^\mu)^2 \chi'(x+\nu\tau) \; dxd\tau
	\leq c.
\end{equation*}
where
\begin{equation*}
c=c(l, \nu, \epsilon, R, T; \|u_0^\mu\|_{H^s}; \|\partial_x^lu_0^\mu\|_{L^2(0,\infty)}; \|u^\mu\|_{L_T^\infty H_x^s}; \|\partial_x^3u^\mu\|_{L_T^1L_x^\infty}).
\end{equation*}
To see that this bound is independent of $\mu>0$, first note
\begin{equation*}
\|u_0^\mu\|_{H^s} \leq \|\widehat{\rho_\mu}\|_\infty \|u_0\|_{H^s} \leq \|u_0\|_{H^s}.
\end{equation*}
As $\chi\equiv0$ for $x<\epsilon$, restricting $0<\mu<\epsilon$ it follows
\begin{equation*}
(\partial_x^lu_0^\mu)^2\chi(x;\epsilon,b)
	= (\rho_\mu \ast \partial_x^lu_01_{[0,\infty)})^2\chi(x;\epsilon,b).
\end{equation*}
Thus by Young's inequality
\begin{align*}
\int_\epsilon^\infty (\partial_x^lu_0^\mu)^2(x) \; dx
	&= \int_\epsilon^\infty (\rho_\mu\ast\partial_x^lu_01_{[0,\infty)})^2(x) \; dx \\
	&\leq \|\rho_\mu\|_1^2 \int_\epsilon^\infty (\partial_x^lu_0)^2(x) \; dx \\
	&\leq \|\partial_x^lu_0\|_{L^2((0,\infty))}^2.
\end{align*}
From Kwon's local well-posedness result \cite{MR2455780} we have
\begin{equation*}
\|u^\mu\|_{L_T^\infty H_x^s} + \|\partial_x^{3}u^\mu\|_{L_T^1L_x^\infty}
	\leq c(\|u_0^\mu\|_{H^s})
	\leq c(\|u_0\|_{H^s})
\end{equation*}
and so we may replace the bound $c=c(\mu)$ with $\tilde{c}$ as in \eqref{PCONST1}.

As the solution depends continuously on the initial data,
\begin{equation*}
\sup_{0 \leq t \leq T} \|u^\mu(t)-u(t)\|_{H^{5/2^+}} \downarrow0
	\quad\text{as}\quad	\mu\downarrow0.
\end{equation*}
Combining this fact with the $\mu$-uniform bound $\tilde{c}$, weak compactness and Fatou's lemma, the theorem holds for all $u_0 \in H^{s}(\mathbb{R})$ with $s>5/2$. This completes 
the proof of Theorem 1 for nonlinearity $u\partial_x^3u$.  

Including nonlinearity $\partial_xu\partial_x^2u$, term $B$ in \eqref{PINEQ} will contain a term
\begin{equation*}
2\int\partial_x^lu\partial_x^l(\partial_xu\partial_x^2u)\chi(x+\nu t) \; dx.
\end{equation*}
As this nonlinearity has a total of three derivatives, integrating by parts produces a form very similar to \eqref{PLB}. The nonlinearity $u^2\partial_xu$, containing only a single derivative, shows no loss of derivatives (see Section \ref{S:7} for a more thorough treatment). 
This completes the proof of Theorem \ref{PROP}.
\end{section}

\begin{section}{Proof of Theorem 2}\label{S:4}

In this section we prove Theorem 2. 
Let $u$ be a smooth solution of IVP \eqref{KDV5G}, differentiate the equation $l$-times and apply \eqref{ENERGY} with $\phi(x,t)=\chi_n(x+\nu t;\epsilon,b)$ to arrive at
\begin{align}
& \frac{d}{dt} \int (\partial_x^lu)^2\chi_n(x+\nu t) \; dx
	+ \int (\partial_x^{l+2}u)^2 \chi_n'(x+\nu t) \; dx \notag \\
	& \qquad\qquad\qquad
		\leq A+B,\label{DINEQ}
\end{align}
where 
\begin{eqnarray*}		
A&=&
\int (\partial_x^lu)^2\left\{\nu\chi_n'(x+\nu t) + \frac32\chi_n^{(5)}(x+\nu t) + \frac{25}{16}\frac{(\chi_n'''(x+\nu t))^2}{\chi_n'(x+\nu t)}\right\} \; dx, \\
B&=&2\int\partial_x^lu\partial_x^l(u\partial_x^3u)
\chi_n(x+\nu t) \; dx.		
\end{eqnarray*}		
The proof proceeds by induction on $l$, however, for fixed $l$ we induct on $n$. The base case $n=0$ coincides with the propagation of regularity result. We invoke constants $c_0,c_1,c_2,\dots,$ depending only on the parameters
\begin{equation} \label{DCONST}
c_k = c_k(n,l; \|u_0\|_{H^s}; \|\partial_x^3u\|_{L_T^1L_x^\infty}; \nu; \epsilon; b; T)
\end{equation}
as well as the decay assumptions on the initial data \eqref{DDATA}.

\underline{Case $l=0$}
Using properties \eqref{CutoffNRatio2} and \eqref{CutoffNDerivatives}, we see
\begin{equation*}
|A| \leq c_0\int u^2(1+\chi_n(x+\nu t)) \; dx.
\end{equation*}
and so integrating in the time interval $[0,t]$, 
we have
\begin{equation} \label{D0A_INEQ}
\left|\int_0^t A \; d\tau\right|
	\leq c_0\left\{T\|u\|_{L_T^\infty L_x^2}^2
		+ \int_0^t \int u^2\chi_n(x+\nu\tau) \; dxd\tau \right\}
\end{equation}
where $0\leq t\leq T$. Additionally,
\begin{equation} \label{D0B_INEQ}
\left|\int_0^t B \; d\tau\right|
	\leq 2\int_0^t\|\partial_x^3u(\tau)\|_{L_x^\infty} \int u^2\chi_n(x+\nu\tau) \; dxd\tau.
\end{equation}
Integrating \eqref{DINEQ} in the time interval $[0,t]$, combining \eqref{D0A_INEQ} and \eqref{D0B_INEQ}, 
we have
\begin{align*}
y(t)
	&:= \int u^2\chi_n(x+\nu t) \; dx
		+ \int_0^t\int(\partial_x^2u)^2\chi_n(x+\nu\tau)\;dxd\tau \\
	&\leq \int u_0^2(x)\chi_n(x) \; dx + \left|\int_0^t A+B \; d\tau\right| \\
	&\leq c_0 + \int_0^t (c_1 + c_2\|\partial_x^3u(\tau)\|_{L_x^\infty}) \int u^2\chi_n(x+\nu\tau) \; dxd\tau \\
	&\leq c_0 + \int_0^t (c_1 + c_2\|\partial_x^3u(\tau)\|_{L_x^\infty}) y(\tau) \; dxd\tau.
\end{align*}
using the hypothesis on the initial data. 
Gronwall's inequality yields 
\begin{equation*}
\sup_{0 \leq t \leq T} \int u^2\chi_n(x+\nu t) \; dx
		+ \int_0^T\int(\partial_x^2u)^2\chi_n(x+\nu\tau)\;dxd\tau
	\leq c_0\exp\left(c_1T+c_2\|\partial_x^3u\|_{L_T^1L_x^\infty}\right).
\end{equation*}
Note that induction in $n$ was not required in this case.
\newline

\underline{Case $l=1$}
Using properties \eqref{CutoffNRatio2} and \eqref{CutoffNDerivatives}, we have
\begin{equation*}
|A| \leq c_0\int (\partial_xu)^2(1+\chi_n(x+\nu t)) \; dx.
\end{equation*}
and so integrating in the time interval $[0,t]$, 
we find 
\begin{equation} \label{D1A_INEQ}
\left|\int_0^t A \; d\tau\right|
	\leq c_0\left\{ T\|u\|_{L_T^\infty H_x^1}^2
		+ \int_0^t \int (\partial_xu)^2\chi_n(x+\nu\tau) \; dxd\tau \right\}
\end{equation}
where $0\leq t\leq T$. After integrating by parts, 
we find
\begin{align}
B
	&= \int \partial_xu(\partial_x^2u)^2 \chi_n(x+\nu t) \; dx + 3 \int u(\partial_x^2u)^2 \chi_n'(x+\nu t) \; dx \notag \\
	&\qquad
		+ \frac43 \int (\partial_xu)^3 \chi_n''(x+\nu t) \; dx
		- \int u(\partial_xu)^2\chi_n'''(x+\nu t) \; dx. \label{D1B}
\end{align}
This expression exhibits a loss of derivatives requiring a correction. A smooth solution $u$ to the IVP \eqref{KDV5G} satisfies the following identity
\begin{align} \label{D1B_KWON}
\lefteqn{\frac{d}{dt} \int u^3\chi_n \; dx}\notag\\
	&= -15\int\partial_xu(\partial_x^2u)^2\chi_n\;dx
		-9\int u(\partial_x^2u)^2\chi_n'\;dx \notag \\
	&\qquad
		+ 10\int(\partial_xu)^3\chi_n'' \; dx
		+ 12\int u(\partial_xu)^2\chi_n''' \; dx
		- \int u^3\chi_n^{(5)} \; dx \notag \\
	&\qquad
		+ 9\int u(\partial_xu)^3\chi_n \; dx
		+ \frac{27}{2} \int u^2(\partial_xu)^2 \chi_n' \; dx
		- \frac34 \int u^4\chi_n''' \; dx \notag \\
	&\qquad
		+ \nu \int u^3 \chi_n' \; dx
\end{align}
after integrating by parts, where $\chi_n^{(j)}$ denotes $\chi_n^{(j)}(x+\nu t)$. Substituting \eqref{D1B_KWON}, we can write \eqref{D1B} as a linear combination of the following terms
\begin{align}
B
	&= \frac{d}{dt} \int u^3 \chi_n \; dx
		+ \int u(\partial_x^2u)^2 \chi_n' \; dx \notag \\
	&\qquad
		+ \int (\partial_xu)^3 \chi_n'' \; dx
		+ \int u(\partial_xu)^2 \chi_n''' \; dx
		+ \int u^3 \chi_n^{(5)} \; dx \notag \\
	&\qquad
		+ \int u(\partial_xu)^3 \chi_n \; dx
		+ \int u^2(\partial_xu)^2 \chi_n' \; dx
		+ \int u^4 \chi_n''' \; dx \notag \\
	&\qquad
		+ \nu \int u^3 \chi_n' \; dx \notag \\
	&=: B_1 + \cdots + B_9.
\end{align}
The fundamental theorem of calculus and 
the Sobolev embedding yield 
\begin{equation}
\left| \int_0^t B_1 \; d\tau \right|
	\leq \|u_0\|_{H^1} \int u_0^2(x) \chi_n(x) \; dx
		+ \|u\|_{L_T^\infty H_x^1} \int u^2 \chi_n(x+\nu t) \; dx
\end{equation}
where $0 \leq t \leq T$. This term is finite by hypothesis \eqref{DDATA} and the case $l=0$. Next,
\begin{equation}
\left|\int_0^t B_2 \; d\tau\right|
	\leq \|u\|_{L_T^\infty H_x^1} \int_0^T\int(\partial_x^2u)^2 \chi_n'(x+\nu\tau) \; dxd\tau,
\end{equation}
which is finite by case $l=0$. Using \eqref{CutoffNPrimeToMinusOne} and the Sobolev embedding, 
we obtain 
\begin{align}
&\left|\int_0^t B_3 + B_4 + B_5 \; d\tau\right| \notag \\
	&\qquad\qquad
		\leq \|u\|_{L_T^\infty H_x^2} \int_0^T\int (\partial_xu)^2|\chi_n''(x+\nu\tau)| + (\partial_xu)^2|\chi_n'''(x+\nu\tau)| \; dxd\tau \notag \\
	&\qquad\qquad\qquad
		+ \|u\|_{L_T^\infty H_x^1} \int_0^T\int u^2|\chi_n^{(5)}(x+\nu\tau)| \; dxd\tau \notag \\
	&\qquad\qquad
		\leq c_0\|u\|_{L_T^\infty H_x^2} \int_0^T\int (\partial_xu)^2\chi_{n-1}(x+\nu\tau;\epsilon/3,b+\epsilon) \; dxd\tau \notag \\
	&\qquad\qquad\qquad
		+ c_1\|u\|_{L_T^\infty H_x^1} \int_0^T\int u^2\chi_{n-1}(x+\nu\tau;\epsilon/3,b+\epsilon) \; dxd\tau.
\end{align}
The first term is finite by induction on $n$ in the current case $l=1$, whereas the second term is finite by the case $l=0$. The Sobolev embedding implies 
\begin{equation}
\left|\int_0^t B_6 \; d\tau\right|
	\leq \|u\|_{L_t^\infty H_x^2}^2 \int_0^t\int (\partial_xu)^2 \chi_n(x+\nu\tau) \; dxd\tau.
\end{equation}
Finally the inequality 
\eqref{CutoffNPrimeToMinusOne} and the Sobolev embedding 
yield 
\begin{equation}
\left|\int_0^t B_7 + B_8 + B_9 \; d\tau\right|
	\leq c_2\|u\|_{L_T^\infty H_x^2}^2 \int_0^T\int u^2\chi_{n-1}(x+\nu\tau;\epsilon/3,b+\epsilon) \; dxd\tau,
\end{equation}
which is finite by case $l=0$. Integrating \eqref{DINEQ} in the time interval $[0,t]$ and combining the above, 
we have 
\begin{align*}
y(t)
	&:= \int (\partial_xu)^2 \chi_n(x+\nu t) \; dx + \int_0^t\int (\partial_x^3u)^2 \chi_n'(x+\nu\tau) \; dxd\tau \\
	&\leq \int (\partial_xu_0)^2(x)\chi_n(x) \; dx + \left|\int_0^t A+B \; d\tau\right| \\
	&\leq c_0 + c_1 \int_0^t \int (\partial_xu)^2 \chi_n(x+\nu\tau) \; dxd\tau \\
	&\leq c_0 + c_1 \int_0^t y(\tau) \; d\tau.
\end{align*}
The result follows by Gronwall's inequality.
\newline

\underline{Cases $l=2,3,4,5$} Due to the structure of the IVP, the cases $l=2,3,4,5$ must be handled individually. The analysis is omitted, however, as it is similar to the cases presented.
\newline

\underline{Case $l\geq6$}
Integrating in the time interval $[0,t]$ and using properties \eqref{CutoffNRatioToMinusOne} and \eqref{CutoffNPrimeToMinusOne}, we have 
\begin{equation} \label{DLA_INEQ}
\left|\int_0^t A \; d\tau\right|
	\leq c_0\int_0^t \int (\partial_x^lu)^2 \chi_{n-1}(x+\nu\tau;\epsilon/3,b+\epsilon) \; dxd\tau,
\end{equation}
which is finite by induction on $n$. Recall \eqref{PLB} and \eqref{PLB1}, wherein we wrote
\begin{equation*}
B = B_{11} + B_{12} + B_2,
\end{equation*}
with the term $B_{11}$ exhibiting a loss of derivatives. Integrating in the time interval $[0,t]$, we  see 
\begin{align} \label{DLB12_INEQ}
\left| \int_0^t B_{12} \; d\tau \right|
	&\leq
		\|u\|_{L_T^\infty H_x^1} \int_0^T\int(\partial_x^{l+1}u)^2 \chi_n'(x+\nu\tau) \; dxd\tau \notag \\
	&\qquad
		+ \int_0^t \|\partial_x^3u(\tau)\|_{L_x^\infty} \int(\partial_x^lu)^2\chi_n(x+\nu\tau) \; dxd\tau \notag \\
	&\qquad
		+ c_0\|u\|_{L_T^\infty H_x^s} \int_0^T\int(\partial_x^lu)^2 \chi_{n-1}(x+\nu\tau) \; dxd\tau
\end{align}
where we have used \eqref{CutoffNPrimeToMinusOne}. The first term is finite by the case $l-1$ and the third is finite by induction on $n$, hence
\begin{equation*}
\left| \int_0^t B_{12} \; d\tau \right|
	\leq c_0 + c_1 \int_0^t \|\partial_x^3u(\tau)\|_{L_x^\infty} \int(\partial_x^lu)^2\chi_n(x+\nu\tau) \; dxd\tau
\end{equation*}

Observe that term $B_2$ only occurs when $l\geq5$. For $l>5$, note that $4+k<l$. 
The inequality \eqref{LinftyTrick} yields 
\begin{align} \label{DLB2}
|B_2|
	&\leq \sum_{k=1}^{\ceil{l/2}-2} c_{l,k} \int|\partial_x^{3+k}u\partial_x^{l-k}u\partial_x^lu|\chi_n(x+\nu t) \; dx \notag \\
	&\leq \int(\partial_x^lu)^2\chi_n(x+\nu t) \; dx \notag \\
	&\qquad
		+ \sum_{k=1}^{\ceil{l/2}-2} \left\{\int(\partial_x^{4+k}u)^2\chi_n(x+\nu t) \; dx
		+ \int(\partial_x^{3+k}u)^2\chi_n(x+\nu t) \; dx\right.\notag \\
	&\qquad\qquad\quad
+\left. \int(\partial_x^{3+k}u)^2\chi_n'(x+\nu t) \; dx\right\}\int(\partial_x^{l-k}u)^2\chi_n(x+\nu t;\epsilon/5,4\epsilon/5) \; dx,
\end{align}
where we have suppressed constants depending on $l$. Integrating in the time interval $[0,t]$, 
we see
\begin{equation} \label{DLB2_INEQ}
\left| \int_0^t B_2 \; d\tau \right|
	\leq c_0 + c_1\int_0^t\int(\partial_x^lu)^2\chi(x+\nu\tau) \; dxd\tau,
\end{equation}
as factors in the summation are estimated via \eqref{CutoffNPrimeToMinusOne} and the inductive hypothesis.

Assuming that $u$ satisfies the IVP \eqref{KDV5G}, we rewrite this term by considering the correction factor
\begin{align}
\frac{d}{dt} \int u(\partial_x^{l-1}u)^2\chi_n(x+\nu t) \; dx 
=\widetilde{C_1}+C_{2}+C_{3}+C_{4},\notag
\end{align}
where
\begin{align*}
\widetilde{C_1}&=
\int\partial_x^5u(\partial_x^{l-1}u)^2\chi_n(x+\nu t) \; dx
		+ 2\int u\partial_x^{l-1}u\partial_x^{l+4}u\chi_n(x+\nu t) \; dx,\\
C_{2}&=
\int u\partial_x^3u(\partial_x^{l-1}u)^2\chi_n(x+\nu t) \; dx,
\\
C_{3}&=
2\int u\partial_x^{l-1}u\partial_x^{l-1}(u\partial_x^3u)\chi_n(x+\nu t) \; dx,\\
C_{4}&=
\nu \int u(\partial_x^{l-1}u)^2\chi_n'(x+\nu t) \; dx.
\end{align*}
Integrating $\widetilde{C_1}$ by parts, 
we have
\begin{equation}
\widetilde{C_1} = \left(\frac{5}{2l-3}\right)B_{11}+C_1,
\end{equation}
where 
\begin{align} \label{DLC3}
C_1
	&= - 5\int u(\partial_x^{l+1}u)^2\chi_n' \; dx
		+ 5\int\partial_x^3u(\partial_x^lu)^2\chi_n \; dx \notag \\
	&\qquad
		+ 15\int\partial_x^2u(\partial_x^lu)^2\chi_n' \; dx
		+ 15\int\partial_xu(\partial_x^lu)^2\chi_n'' \; dx
		\notag \\
	&\qquad	
		+ 5\int u(\partial_x^lu)^2\chi_n''' \; dx 
		- 5\int\partial_x^4u(\partial_x^{l-1}u)^2\chi_n' \; dx
		\notag \\
	&\qquad	
	    - 10\int\partial_x^3u(\partial_x^{l-1}u)^2\chi_n'' \; dx
- 10\int\partial_x^2u(\partial_x^{l-1}u)^2\chi_n''' \; dx \notag \\
	&\qquad
		- 5\int\partial_xu(\partial_x^{l-1}u)^2\chi_n^{(4)} \; dx
		- \int u(\partial_x^{l-1}u)^2\chi_n^{(5)} \; dx.
\end{align}
Here $\chi_n^{(j)}$ denotes 
$\chi_n^{(j)}(x+\nu t;\epsilon,b)$. 
The fundamental theorem of calculus yields 
\begin{eqnarray*}
\lefteqn{\left(\frac{5}{2l-3}\right)\left|\int_0^t B_{11} \; d\tau\right|}\\
	&\leq& \left|\int u_0(\partial_x^{l-1}u_0)^2\chi_n(x) \; dx\right|
		+ \left|\int u(\partial_x^{l-1}u)^2\chi_n(x+\nu t) \; dx\right| \\
	& &\qquad
		+ \left|\int_0^t C_1 + C_2 + C_3 + C_4 \; d\tau\right|.
\end{eqnarray*}
We now concern ourselves with estimating the right-hand side of this expression. First note
\begin{align}
&\left|\int u_0(\partial_x^{l-1}u_0)^2\chi_n(x) \; dx\right|
		+ \left|\int u(\partial_x^{l-1}u)^2\chi_n(x+\nu t) \; dx\right| \notag \\
	&\qquad\qquad\leq
		\|u_0\|_{H^1} \|x^{n/2}\partial_x^{l-1}u_0\|_{L_x^2(\epsilon,\infty)}^2
		+ \|u\|_{L_T^\infty H_x^1} \int (\partial_x^{l-1}u)^2\chi_n(x+\nu t) \; dx,
\end{align}
is bounded by the hypothesis \eqref{DDATA} and the case $l-1$. Similarly to $B_2$ and $B_{12}$, integrating in the time interval $[0,t]$, using \eqref{LinftyTrick} and property \eqref{CutoffNPrimeToMinusOne}, we obtain
\begin{equation}
\left| \int_0^t C_1 \; d\tau \right|
	\leq c_0 + \int_0^t (c_1+c_2\|\partial_x^3u(\tau)\|_{L_x^\infty})\int (\partial_x^lu)^2 \chi_n(x+\nu\tau) \; dxd\tau
\end{equation}
where the term containing $(\partial_x^{l+1}u)^2\chi_n'$ is controlled using the induction case $l-1$, as in \eqref{DLB12_INEQ}.

Using \eqref{LinftyTrick} and the inductive hypothesis, 
we see 
\begin{equation}
\left| \int_0^t C_2 \; d\tau \right| \leq c_0,
\end{equation}
similar to $B_2$. The same technique applies to $C_3$ and $C_4$.

Integrating \eqref{DINEQ} in the time interval $[0,t]$ and combining the above, we find that 
there exists constants as in \eqref{DCONST} such that
\begin{align*}
y(t)
	&:= \int (\partial_x^lu)^2 \chi_n(x+\nu t) \; dx + \int_0^t\int (\partial_x^{l+2}u)^2 \chi_n'(x+\nu\tau) \; dxd\tau \\
	&\leq \int (\partial_x^lu_0)^2(x)\chi_n(x) \; dx + \left|\int_0^t A+B \; d\tau\right| \\
	&\leq c_0 + \int_0^t (c_1+c_2\|\partial_x^3u(\tau)\|_{L_x^\infty}) \int (\partial_x^lu)^2 \chi_n(x+\nu\tau) \; dxd\tau \\
	&\leq c_0 + \int_0^t (c_1+c_2\|\partial_x^3u(\tau)\|_{L_x^\infty}) y(\tau) \; d\tau.
\end{align*}
The result follows by Gronwall's inequality. To handle the case of arbitrary data $u_0 \in H^{s}(\mathbb{R})$ 
with $s>5/2$, a limiting argument similar to the proof of Theorem \ref{PROP} is used. This completes the proof of Theorem 2. 
\end{section}

\begin{section}{Proof of Theorem 3}\label{S:5}

In this section we prove  Theorem 3. 
Integration by parts yields the next lemma.

\begin{lemma} \label{DIR_LEMMA}
Suppose for some $l\in\mathbb{Z}^+$
\begin{equation}
\sup_{0 \leq t \leq T}
	\int (\partial_x^lu)^2 \chi_n(x+\nu t) \; dx
	+ \int_0^T\int (\partial_x^{l+2}u)^2 \chi_n'(x+\nu\tau) \; dxd\tau < \infty.
\end{equation}
Then for every $0<\delta<T$, there exists $\hat{t}\in(0,\delta)$ such that
\begin{equation}
\int (\partial_x^{l+j}u)^2 \chi_{n-1}(x+\nu\hat{t};\epsilon^+,b) \; dx < \infty
	\qquad (j=0,1,2).
\end{equation}
\end{lemma}

To prove Theorem 3, it suffices to consider an example; fix $n=9$ in the hypothesis of the theorem. Then we may apply Theorem \ref{DECAY} with $(l,n)=(0,9)$. Thus, after applying Lemma \ref{DIR_LEMMA}, there exists $t_0\in(0,\delta/2)$ such that
\begin{equation*}
\int (u^2 + (\partial_xu)^2 + (\partial_x^2u)^2) \chi_8(x+\nu t_0;\epsilon^+,b) \; dx < \infty.
\end{equation*}
Hence we may apply Theorem \ref{DECAY} with $(l,n)=(2,8)$ and find $t_1 \in(t_0,\delta/2)$ such that
\begin{equation*}
\int (u^2 + \cdots + (\partial_x^4u)^2) \chi_7(x+\nu t_1;\epsilon^+,b) \; dx < \infty.
\end{equation*}
Continuing in this manner, applying Theorem \ref{DECAY} with $(l,n)=(4,7),(6,6),\dots,(18,0)$ provides the existince of $\hat{t}\in(\delta/2,\delta)$ such that
\begin{equation*}
\int (u^2 + \cdots + (\partial_x^{19}u)^2) \chi(x+\nu\hat{t};\epsilon^+,b) \; dx < \infty.
\end{equation*}
Finally, we can apply Theorem \ref{PROP} with $l=19$, completing the proof.
\end{section}

\begin{section}{Proof of Corollary \ref{NU}}
\label{S:6}

The proof of Corollary \ref{NU} relies on the following lemma, which follows by considering a dyadic decomposition of the interval $[0,\infty)$. Observe that the lemma also applies when integrating a nonnegative function on the interval $[-(a+\epsilon),-\epsilon]$, implying decay on the left half-line.
\begin{lemma} \label{NU_LEMMA}
Let $f:[0,\infty) \rightarrow [0,\infty)$ 
be continuous. If for $a>0$
\begin{equation*}
\int_0^a f(x) \; dx \leq ca^\alpha
\end{equation*}
then for every $\varepsilon>0$\
\begin{equation*}
\int_0^\infty \frac{1}{\langle x \rangle^{\alpha+\varepsilon}} f(x) \; dx \leq c(\alpha,\varepsilon).
\end{equation*}
\end{lemma}

Now we prove Corollary \ref{NU}.
\begin{proof}
Recall that for $l\geq6$, Theorem \ref{PROP} with $x_0=0$ states
\begin{equation*}
\sup_{0 \leq t \leq T} \int_{\epsilon-\nu t}^\infty (\partial_x^lu)^2(x,t) \; dx \leq c(\nu;8(l-5)).
\end{equation*}
For fixed $t\in(0,T)$
\begin{equation*}
\int_{\epsilon-\nu t}^\infty (\partial_x^lu)^2(x,t) \; dx
	= \left(\int_{\epsilon-\nu t}^\epsilon + \int_\epsilon^\infty\right)(\partial_x^lu)^2(x,t) \; dx
	:= I + II.
\end{equation*}
Theorem \ref{PROP} with $\nu=0$ yields control of $II$, so we focus on $I$. For $\nu^*$ large enough, $\nu>\nu^*$ implies
\begin{equation*}
I = \int_{\epsilon-\nu t}^\epsilon (\partial_x^lu)^2(x,t) \; dx \leq ct^{-8(l-5)}(\nu t)^{8(l-5)}.
\end{equation*}
Applying Lemma \ref{NU_LEMMA} with $a=\nu t$ and $\alpha=8(l-5)$, we find
\begin{equation*}
\int_{-\infty}^\epsilon \frac{1}{\langle x \rangle^{8(l-5)+\varepsilon}} (\partial_x^lu)^2(x,t) \; dx < \infty
\end{equation*}
for $\varepsilon>0$.
This completes the proof of Corollary \ref{NU}.
\end{proof}

\end{section}

\begin{section}{Extensions to Other Models}\label{S:7}

In this section we prove the following extension of Theorem \ref{PROP}, which applies to those equations described by Theorem A.
\begin{theorem} \label{PROPE}
Consider the class of initial value problems
\begin{equation} \label{KDV5E}
\begin{cases}
	\partial_tu - \partial_x^5u + Q(u,\partial_xu,\partial_x^2u,\partial_x^3u) = 0,
		\qquad x,t\in\mathbb{R}, \\
	u(x,0) = u_0(x),
\end{cases}
\end{equation}
where $Q:\mathbb{R}^4\rightarrow\mathbb{R}$ is a polynomial having no constant or linear terms. Let $u$ be a solution to IVP \eqref{KDV5E} satisfying
\begin{equation*}
u \in C([-T,T] ; X_{s,m}),
	\quad m\in\mathbb{Z}, s\in\mathbb{R},
\end{equation*}
such that $m \geq m_0$ and $s \geq \max\{s_0,2m\}$ for a nonnegative integer $m_0$ and positive real number $s_0$ determined by the form of the nonlinearity $Q$. If $u_0 \in X_{s,m}$ additionally satisfies
\begin{equation}
\|\partial_x^lu_0\|_{L^2(x_0,\infty)}^2
	= \int_{x_0}^\infty (\partial_x^lu_0)^2(x) \; dx
	< \infty,
\end{equation}
for some $l \in \mathbb{Z}^+, x_0 \in \mathbb{R}$, then $u$ satisfies
\begin{equation}
\sup_{0 \leq t \leq T} \int_{x_0 + \epsilon - \nu t}^\infty (\partial_x^k u)^2(x,t) \; dx \leq c
\end{equation}
for any $\nu\geq0, \epsilon>0$ and each $k=0,1,\dots,l$ with
\begin{equation} \label{PECONST1}
c = c(l; \nu; \epsilon; T; \|u_0\|_{X_{s,m}}; \|\partial_x^lu_0\|_{L^2(x_0,\infty)}).
\end{equation}
Moreover, for any $\nu\geq0, \epsilon>0$ and $R>\epsilon$
\begin{equation}
\int_0^T \int_{x_0 + \epsilon - \nu t}^{x_0 + R - \nu t} (\partial_x^{l+2}u)^2(x,t) \; dxdt \leq \tilde{c}
\end{equation}
with
\begin{equation} \label{PECONST2}
\tilde{c} = \tilde{c}(l; \nu; \epsilon; R; T; \|u_0\|_{X_{s,m}}; \|\partial_x^lu_0\|_{L^2(x_0,\infty)}).
\end{equation}
\end{theorem}

\begin{remark}
Due to the similarities in the proof technique, the comments in this section can be modified to prove extensions of Theorems \ref{DECAY} and \ref{DIR} to the class \eqref{KDV5E}.
\end{remark}
\begin{remark}
Establishing local well-posedness of the IVP \eqref{KDV5E} in the weighted Sobolev spaces $X_{s,m}$ imposes minimum values on $m$ and $s$, see for instance the contraction principle technique used by Kenig, Ponce and Vega in \cite{MR1321214} and \cite{MR1195480}. Thus the values of $m_0$ and $s_0$ are determined by considering both the local well-posedness as well as our proof of the propagation of regularity. As we see below, these considerations may differ.
\end{remark}
\begin{remark}
A slight modification to the energy inequality \eqref{ENERGY} allows one to loosen the restriction that $Q$ not contain any linear terms. In particular, the theorem applies to the model \eqref{E2} when coupled with an appropriate local well-posedness theorem. Provided suitable cutoff functions exist, modifications to \eqref{ENERGY} also extend the technique to a class of higher order equations containing the KdV heirarchy.
\end{remark}

\begin{proof}
Though not strictly necessary, we break the proof into cases based on the form of the nonlinearity $Q(u)$. We treat the case $x_0=0$ as the argument is translation invariant. Following the proof of Theorem \ref{PROP}, let $u$ be a smooth solution of the IVP \eqref{KDV5E}. Differentiating the equation $l$-times, applying \eqref{ENERGY} and using properties of $\chi$, we arrive at
\begin{align} \label{PE_INEQ}
& \frac{d}{dt} \int (\partial_x^lu)^2\chi(x+\nu t) \; dx
	+ \int (\partial_x^{l+2}u)^2 \chi'(x+\nu t) \; dx \notag \\
	& \qquad\qquad\qquad
		\lesssim \int (\partial_x^lu)^2\chi'(x+\nu t;\epsilon/3,b+\epsilon) \; dx
			+ \int\partial_x^lu\partial_x^lQ(u)\chi(x+\nu t) \; dx \notag \\
	& \qquad\qquad\qquad
		=: A+B
\end{align}
The proof proceeds by induction on $l\in\mathbb{Z}^+$. For a given nonlinearity $Q(u)$, there exists $l_0\in\mathbb{Z}^+$ such that the cases $l=0,1,\dots,l_0$ can be proved by choosing $s_0$ large enough. Thus it suffices to prove only the inductive step. We describe the formal calculations, omitting the limiting argument.

Integrating in the time interval $[0,t]$ and applying the $l-2$ result we have
\begin{equation} \label{PEA_INEQ}
\left|\int_0^t A \; d\tau\right|
	\leq c(\nu;\epsilon;b)\int_0^T \int (\partial_x^lu)^2\chi'(x+\nu\tau) \; dxd\tau
	\leq c_0
\end{equation}
where $0\leq t\leq T$ and $c_0$ as in \eqref{PECONST1}. We now turn to term $B$.
\newline

\underline{Case 1}
Suppose $Q$ is independent of both $\partial_x^2u$ and $\partial_x^3u$. Then there exists $N\in\mathbb{Z}^+$ such that, after integrating by parts, $B$ is a linear combination of terms of the form
\begin{equation*}
\int u^{j_0}(\partial_xu)^{j_1}(\partial_x^2u)^{j_2}(\partial_x^lu)^2 \chi(x+\nu t) \; dx,
	\qquad j_0,j_1,j_2 \leq N,
\end{equation*}
and
\begin{equation*}
\int u^{j_0}(\partial_xu)^{j_1}(\partial_x^2u)^{j_2}(\partial_x^ku)^2 \chi^{(j_3)}(x+\nu t) \; dx,
	\qquad j_0,j_1,j_2 \leq N
\end{equation*}
where $1 \leq j_3 \leq 5$ and $3 \leq k \leq l+1$.
Hence no loss of derivatives occurs. Integrating in the time interval $[0,t]$, applying the induction hypothesis and the Sobolev embedding
\begin{equation*}
\left|\int_0^t B \; d\tau\right| \leq c_0 + c_1 \int_0^t\int(\partial_x^lu)^2\chi(x+\nu\tau) \; dxd\tau
\end{equation*}
provided $s_0>7/2$, with $c_0$ and $c_1$ as in \eqref{PECONST1}. Combining with \eqref{PEA_INEQ}, after integrating \eqref{PE_INEQ} in time and using the hypothesis on the initial data we have
\begin{align}
y(t)
	&:= \int (\partial_x^lu)^2\chi(x+\nu t) \; dx
		+ \int_0^t\int (\partial_x^{l+2}u)^2 \chi'(x+\nu\tau) \; dxd\tau \notag \\
	&\leq c_0 + c_1\int_0^t\int(\partial_x^lu)^2\chi(x+\nu\tau) \; dxd\tau \notag \\
	&\leq c_0 + c_1\int_0^ty(\tau)\;d\tau.
\end{align}
The result follows by an application of Gronwall's inequality. The value of $m_0$ is determined by the LWP theory.
\newline

\underline{Case 2}
Suppose $Q$ is a linear combination of quadratic terms (with the exception of $u\partial_x^2u$). After integrating by parts $B$ is a linear combination of terms of the form
\begin{equation*}
\int\partial_x^ju(\partial_x^{l+1}u)^2 \chi(x+\nu t) \; dx,
	\qquad 1 \leq j \leq 4
\end{equation*}
as well as lower order terms. The correction technique of Theorem \ref{PROP} can be modified to account for this loss of derivatives. For example, if $Q(u)=\partial_x^2u\partial_x^3u$, then integrating by parts and supressing coefficients
\begin{equation*}
B = \int\partial_x^2u(\partial_x^{l+1}u)^2 \chi(x+\nu t) \; dx + \int\partial_x^4u(\partial_x^lu)^2 \chi(x+\nu t) \; dx + \tilde{B}
\end{equation*}
where $\tilde{B}$ is controlled by induction. For the second term, we impose $s_0>9/2$ to control $\|\partial_x^4u\|_{L_x^\infty}$. For the first term, consider the correction
\begin{equation*}
\frac{d}{dt} \int\partial_xu(\partial_x^{l-1}u)^2\chi(x+\nu t) \; dx.
\end{equation*}
In general, more than one correction may be necessary. The remainder of the proof is similar to Theorem \ref{PROP}, thus the value of $m_0$ is determined by the LWP theory. Note that if $Q$ additionally contained higher degree terms independent of $\partial_x^2u$ and $\partial_x^3u$, the above argument applies. Equations in the class \eqref{KDV5G} are of this form.
\newline

\underline{Case 3}
The remaining nonlinearities in the class \eqref{KDV5E} exhibit a loss of derivatives which, in general, cannot be controlled by the correction technique. We illustrate the argument in this case by focusing on the example equation
\begin{equation} \label{KDVXX}
\partial_tu - \partial_x^5u = u\partial_x^2u.
\end{equation}
The IVP associated to this equation is locally well-posed in $H^s(\mathbb{R}), s\geq2$, using the contraction mapping principle. However, our modification to the proof of Theorem \ref{PROP} will require the use of weighted Sobolev spaces.

After integrating by parts and supressing coefficients
\begin{equation} \label{BXX}
B = \int u(\partial_x^{l+1}u)^2\chi(x+\nu t) \; dx + \int\partial_x^2u(\partial_x^lu)^2\chi(x+\nu t) \; dx + \tilde{B}
\end{equation}
where $\tilde{B}$ is controlled by induction. Combining with \eqref{PEA_INEQ}, after integrating \eqref{PE_INEQ} in time and using the hypothesis on the initial data we have
\begin{align}
y(t)
	&:= \int (\partial_x^lu)^2\chi(x+\nu t) \; dx + \int_0^t\int (\partial_x^{l+2}u)^2 \chi'(x+\nu\tau) \; dxd\tau \notag \\
	&\leq c_0 + \int_0^t\int\partial_x^2u(\partial_x^lu)^2\chi(x+\nu\tau)\;dxd\tau
		+ \left|\int_0^t\int u(\partial_x^{l+1}u)^2\chi(x+\nu\tau)\;dxd\tau\right| \notag \\
	&\leq c_0 + c_1\int_0^t y(\tau) \; d\tau
		+ \left|\int_0^t\int u(\partial_x^{l+1}u)^2\chi(x+\nu\tau)\;dxd\tau\right|.
\end{align}
Focusing on the last term in the above line,
\begin{align} \label{LOSSXX}
& \left|\int_0^t\int u(\partial_x^{l+1}u)^2\chi(x+\nu\tau)\;dxd\tau\right| \notag \\
	&\qquad\qquad\le
		\left(\sum_{j\in\mathbb{Z}} \sup_{\substack{0 \leq t \leq T \\ j\leq x \leq j+1}} |u(x,t)| \right)
			\left(\sup_{j\in\mathbb{Z}} \int_0^T\int_j^{j+1}(\partial_x^{l+1}u)^2\chi(x+\nu\tau)\;dxd\tau \right).
\end{align}
We check three cases to show the inductive case $l-1$ bounds the second factor. First, the integral vanishes for $j+1<\epsilon-\nu T$. For $\epsilon<j$ we apply the inductive hypothesis with $\nu=0$. Otherwise we utilize a pointwise bound on $\chi$
\begin{equation*}
\int_0^T\int_j^{j+1}(\partial_x^{l+1}u)^2\chi(x+\nu\tau)\;dxd\tau
	\lesssim \int_0^T\int(\partial_x^{l+1}u)^2\chi'(x+\nu\tau;\epsilon/5,\nu T+ \epsilon)\;dxd\tau.
\end{equation*}

The technique for bounding the first factor is described in the next theorem. In general, there exists a nonnegative integer $n$ depending on the form of the polynomial $Q$ such that the following quantities must be estimated:
\begin{equation*}
\sum_{j\in\mathbb{Z}} \sup_{\substack{0 \leq t \leq T \\ j\leq x \leq j+1}} |\partial_x^ku(x,t)|,
	\quad k=0,1,\dots,n,
\end{equation*}
assuming $u$ is a Schwarz solution of IVP \eqref{KDV5E}. With such an estimate in hand, the result follows by an application of Gronwall's inequality.
\end{proof}

\begin{theorem} \label{SUM}
Let $k\in\mathbb{Z}^+\cup\{0\}$ and $u$ be a Schwartz solution of the IVP \eqref{KDV5E} corresponding to initial data $u_0\in\mathscr{S}(\mathbb{R})$. Then there exists a nonnegative integer $m_0$ (depending on $Q$ and $k$) and positive real number $s_0\geq2m_0$ such that
\begin{equation*}
\sum_{j\in\mathbb{Z}} \sup_{\substack{0 \leq t \leq T \\ j\leq x \leq j+1}} |\partial_x^ku(x,t)|
	\leq c(T;\|u_0\|_{X_{s_0,m_0}}).
\end{equation*}
\end{theorem}

The idea is to apply a Sobolev type inequality in the $t$-variable and show that the resulting summation converges by imposing enough spatial decay on the solution. Acheiving this goal requires the following lemma.

\begin{lemma} \label{SOB2XX}
If $f \in C^2(\mathbb{R}^2)$, then
\begin{align*}
\sup_{\substack{0 \leq t \leq T \\ 0 \leq x \leq L}} |f(x,t)|
	&\leq \int_0^T\int_0^L |\partial_{xt} f(y,s)| \; dyds + \frac{1}{TL} \int_0^T\int_0^L |f(y,s)| \; dyds \\
	&\qquad\qquad
		\frac1L \int_0^T\int_0^L |\partial_tf(y,s)| \; dyds + \frac1T \int_0^T\int_0^L |\partial_xf(y,s)| \; dyds
\end{align*}
for any $L,T>0$.
\end{lemma}

We now turn to the proof of Theorem \ref{SUM}.
\begin{proof}
For concreteness, we show details for $k=0$. Applying Lemma \ref{SOB2XX},
\begin{equation*}
\sum_{j\in\mathbb{Z}} \sup_{\substack{0 \leq t \leq T \\ j\leq x \leq j+1}} |u(x,t)|
	\lesssim_T \|\partial_{xt}u\|_{L_T^1 L_x^1} + \|\partial_xu\|_{L_T^1 L_x^1} + \|\partial_tu\|_{L_T^1 L_x^1} + \|u\|_{L_T^1 L_x^1}.
\end{equation*}
Focusing on the worst term $\|\partial_{xt}u\|_{L_T^1L_x^1}$ and applying
\begin{equation*}
\|f\|_1 \leq \|f\|_2 + \|xf\|_2
\end{equation*}
we arrive at
\begin{equation*}
\|\partial_{xt}u\|_{L_T^1L_x^1}
	\lesssim_T \|\partial_{xt}u\|_{L_T^\infty L_x^2} + \|x\partial_{xt}u\|_{L_T^\infty L_x^2}.
\end{equation*}
Looking at the second term and using the differential equation we have
\begin{equation*}
\|x\partial_{xt}u\|_2
	\leq \|x\partial_x^6u(t)\|_2
		+ \|x\partial_x(u\partial_x^2u)\|_2
	=: A+ B.
\end{equation*}
Then
\begin{align*}
A^{2}
	&= \int x^{2}(\partial_x^6u)^2dx\\
	&= \int u\partial_x^6(x^2\partial_x^6u)dx\\
	&= \int x^{2}u\partial_x^{12}udx
		+12\int xu\partial_x^{11}udx
		+30\int u\partial_x^{10}udx\\
	&\lesssim \|x^{2}u\|_2\|\partial_x^{12}u\|_2
		+ \|xu\|_2\|\partial_x^{11}u\|_2
		+ \|u\|_2\|\partial_x^{10}u\|_2.
\end{align*}
and so we impose $s_0\geq12, m_0\geq4$ (compared to the $H^2(\mathbb{R})$ local well-posedness). The estimates for the remaining terms are similar, completing the case $k=0$.
\end{proof}
\end{section}

\vskip3mm
\noindent {\bf Acknowledgments.}
A portion of this work was completed while J.S was visiting
the Department of Mathematics at the University of
California, Santa Barbara whose hospitality he gratefully
acknowledges. The authors thank Professor Gustavo Ponce
for giving us valuable comments.
J.S is partially supported by JSPS, Strategic Young Researcher Overseas
Visits Program for Accelerating Brain Circulation and by MEXT,
Grant-in-Aid for Young Scientists (A) 25707004.



\end{document}